\documentclass{article}

\usepackage[top=1in,bottom=1in,left=1in,right=1in]{geometry}
\usepackage{graphicx}
\usepackage{caption}
\usepackage{subcaption}
\usepackage{float} 
\usepackage{lmodern}
\usepackage[T1]{fontenc}
\usepackage[english,activeacute]{babel}
\usepackage{mathtools}
\usepackage{amsmath,amsthm,amssymb,amsfonts}
\usepackage{mathrsfs}
\usepackage{stmaryrd} 
\usepackage[noadjust]{cite} 
\usepackage{enumitem}
\setlist[enumerate]{noitemsep}
\setlist[itemize]{leftmargin=*,noitemsep}
\usepackage{verbatim} 
\usepackage[usenames,dvipsnames,svgnames,table]{xcolor}
\usepackage[utf8]{inputenc} 
\usepackage{soul} 
\usepackage{tikz-cd} 
\usetikzlibrary{babel} 
\usepackage{setspace} 
\usepackage{array, multicol, multirow, makecell}
\usepackage{subfiles}
\usepackage[nottoc]{tocbibind} 
\usepackage{longtable}
\usepackage{tikz}
\usetikzlibrary{positioning}
\usepackage{dynkin-diagrams}
\usepackage{authblk}
\usepackage{abstract}
\usepackage{bbm}

\usepackage[pdftex,backref=page]{hyperref} 

\hypersetup{linktocpage=true,colorlinks=true,citecolor=blue,urlcolor=blue}
\newcommand{\doi}[1]{\href{https://doi.org/#1}{\ttfamily #1}}

\usepackage{orcidlink}

\usepackage{import}
\usepackage{xifthen}
\usepackage{pdfpages}
\usepackage{transparent}

\let\OLDthebibliography\thebibliography
\renewcommand\thebibliography[1]{
	\OLDthebibliography{#1}
	\setlength{\parskip}{0pt}
}

\newtheorem{theorem}{Theorem}[section]
\newtheorem{corollary}[theorem]{Corollary}
\newtheorem{proposition}[theorem]{Proposition}
\newtheorem{lemma}[theorem]{Lemma}
\newtheorem{conjecture}{Conjecture}
\newtheorem{question}{Question}
\theoremstyle{definition}
\newtheorem{definition}[theorem]{Definition}

\newtheorem{remark}[theorem]{Remark}

\DeclareMathOperator{\Id}{Id}

\DeclareMathOperator{\End}{End}

\DeclareMathOperator{\ad}{ad}
\DeclareMathOperator{\Der}{Der}

\DeclareMathOperator{\Trace}{tr}

\DeclareMathOperator{\Rm}{Rm}
\DeclareMathOperator{\Ric}{Ric}
\DeclareMathOperator{\ric}{ric}

\DeclareMathOperator{\Heis}{Heis}
\DeclareMathOperator{\diag}{diag}

\DeclareMathOperator{\CH}{\mathbb{C}\mathrm{H}}
\DeclareMathOperator{\HH}{\mathbb{H}\mathrm{H}}

\newcommand{\N}{\mathbb{N}}

\newcommand{\R}{\mathbb{R}}
\newcommand{\C}{\mathbb{C}}

\newcommand{\calB}{\mathcal{B}}

\newcommand{\calL}{\mathcal{L}}

\newcommand{\calQ}{\mathcal{Q}}

\newcommand{\fra}{\mathfrak{a}}
\newcommand{\frb}{\mathfrak{b}}
\newcommand{\frg}{\mathfrak{g}}

\newcommand{\frl}{\mathfrak{l}}

\newcommand{\frn}{\mathfrak{n}}

\newcommand{\frs}{\mathfrak{s}}

\renewcommand{\Re}{\operatorname{Re}}
\renewcommand{\Im}{\operatorname{Im}}

\newcommand{\SU}{\mathrm{SU}}
\newcommand{\Sp}{\mathrm{Sp}}

\newcommand{\heis}{\mathfrak{heis}}

\newcommand{\pd}[2]{\frac{\partial #1}{\partial #2}}
\renewcommand{\d}{\mathrm{d}}

\DeclarePairedDelimiter{\abs}{\lvert}{\rvert}
\DeclarePairedDelimiter{\norm}{\lVert}{\rVert}

\allowdisplaybreaks

\title{\textbf{Quaternionic K\"ahler manifolds fibered by solvsolitons}}
\author{Vicente Cortés, Alejandro Gil-García \orcidlink{0000-0002-9370-241X}, and Markus Röser}
\affil{\normalsize Fachbereich Mathematik\\
    Universit\"at Hamburg\\
    Bundesstra\ss e 55, 20146 Hamburg, Germany\\
    vicente.cortes@uni-hamburg.de, markus.roeser@uni-hamburg.de}
\affil{\normalsize Beijing Institute of Mathematical Sciences and Applications (BIMSA)\\
    No.\ 544, Hefangkou Village, Huaibei Town, Huairou District, Beijing 101408, China\\
    alejandrogilgarcia@bimsa.cn}
\date{\today}

\begin{document}


\maketitle
  
\begin{abstract}

This paper is concerned with the geometry of principal orbits in quaternionic K\"ahler manifolds $M$ of cohomogeneity one. We focus on the complete cohomogeneity one examples obtained from the non-compact quaternionic K\"ahler symmetric spaces associated with the simple Lie groups of type A by the one-loop deformation. We prove that for zero deformation parameter the principal orbits form a fibration by solvsolitons (nilsolitons if $4n=\dim M=4$). The underlying solvable group is non-unimodular if $n>1$ and is the Heisenberg group if $n=1$. We show that under the deformation, the hypersurfaces remain solvmanifolds but cease to be Ricci solitons.\bigskip

\medskip


\end{abstract}

\clearpage

\tableofcontents


\section{Introduction}

Quaternionic Kähler manifolds are $4n$-dimensional Riemannian manifolds whose holonomy group is contained in $\Sp(n)\Sp(1)$ (restricting to $n>1$, for the moment). The latter group belongs to Berger's list of irreducible Riemannian holonomy groups \cite{Ber55}. Equivalently, quaternionic Kähler manifolds are Riemannian manifolds $(M,g)$ of non-zero scalar curvature equipped with a parallel skew-symmetric almost quaternionic structure $\calQ\subset\End(TM)$. They are Einstein manifolds, so their theory naturally divides in positive or negative scalar curvature. Positive quaternionic Kähler geometry is very rigid, and it is conjectured that the only examples of complete quaternionic Kähler manifolds of positive scalar curvature are symmetric spaces of compact type \cite{LS94}.\medskip

On the other hand, negative quaternionic Kähler geometry has shown to be very rich. There exist locally symmetric, homogeneous non-symmetric and complete non-locally homogeneous examples. Recently, examples of complete non-locally homogeneous quaternionic Kähler manifolds of negative scalar curvature with two ends, one of finite volume and the other one of infinite volume, have been constructed in all dimensions \cite{CRT21,Cor23}. Nevertheless, the problem of finding complete non-locally symmetric quaternionic Kähler manifolds of finite volume is still open, in contrast with all the other holonomy groups in Berger list, in which even compact non-locally symmetric examples are known.\medskip

A way to obtain quaternionic Kähler manifolds of negative scalar curvature is through the so-called \emph{supergravity $c$-map} construction \cite{CFG89,FS90,Hit09}. This was originally introduced by physicists and assigns a quaternionic Kähler manifold of negative scalar curvature to each \emph{projective special Kähler manifold}. Moreover, the supergravity $c$-map metric admits a one-parameter deformation by quaternionic Kähler metrics \cite{RSV06,ACDM15}. This metric is usually known as the \emph{(one-loop) deformed supergravity $c$-map metric}. Many of the known examples of quaternionic Kähler manifolds of negative scalar curvature are in the image of the supergravity $c$-map. In particular, all homogeneous examples, except the quaternionic hyperbolic spaces $\HH^n$, are supergravity $c$-map spaces \cite{dWVP92}. Hence, they admit a one-parameter deformation which depends on a real constant $c\in\R$. When $c>0$, they are complete \cite{CDS17} and (exactly) of cohomogeneity one \cite{CST21,CGS23}.\medskip

Thanks to the results of \cite{Ale75,Cor96,Heb98,Lau10,BL23}, it is known that every homogeneous quaternionic Kähler  manifold of negative scalar curvature is an \emph{Alekseevsky space}, that is, homogeneous under a simply transitive completely solvable group of isometries. Therefore, every homogeneous quaternionic Kähler  manifold of negative scalar curvature is an \emph{Einstein solvmanifold}. The systematic study of Einstein solvmanifolds goes back to the seminal work of Heber \cite{Heb98} and since then many important results have been obtained (see for instance \cite{Lau01,Lau09,Lau11,LL14,BL22,BL23}). An important feature of Einstein solvmanifolds is that it is possible to construct \emph{algebraic Ricci solitons} on some of their hypersurfaces. They are a specialization of the Ricci solitons introduced by Hamilton \cite{Ham88} adapted to the algebraic structure of the homogeneous manifold. More precisely, an algebraic Ricci soliton on a solvable/nilpotent Lie group is called \emph{solvsoliton/nilsoliton}.\medskip

As we have explained, the one-loop deformation of every homogeneous supergravity $c$-map space is of cohomogeneity one. In this paper we study the induced geometry of the hypersurface orbits of the group acting with cohomogeneity one. In particular, we focus on the $4n$-dimensional quaternionic Kähler manifold $(\bar N,g_{\bar N}^c)$, where 
\begin{equation}\label{eq:SU_sym_space}
    \bar N:=\frac{\mathrm{SU}(n,2)}{\mathrm{S}(\mathrm{U}(n)\times\mathrm{U}(2))}
\end{equation}
is equipped with its one-loop deformed metric $g_{\bar N}^c$. This quaternionic Kähler manifold is obtained by applying the deformed supergravity $c$-map to $\CH^{n-1}$ considered as a homogeneous projective special Kähler manifold (see \cite{CDS17}). If $c=0$, then $(\bar N,g^0_{\bar N})$ is the quaternionic K\"ahler symmetric space $(\SU(n,2)/\mathrm{S}(\mathrm{U}(n)\times\mathrm{U}(2)),g_{\mathrm{can}})$. If $c>0$ then $(\bar N,g^c_{\bar N})$ is of cohomogeneity one. The manifold $(\bar N, g_{\bar N}^c)$ may be described as a product $(0,\infty)\times K$ for an open subset $K\subset\R^{4n-1}$ and the metric has the form
$$g_{\bar{N}}^c=\frac{1}{4\rho^2}\frac{\rho+2c}{\rho+c}\d\rho^2+g_\rho^c,$$
where $g_\rho^c:=g_{\bar N}^c|_{T\bar N_\rho\times T\bar N_\rho}$ is the induced metric on the fiber $\bar N_\rho := \{\rho\}\times K$. As described explicitly in \cite{CRT21}, the fibers $\{\rho\}\times K$ are the orbits of a $c$-dependent isometric action of the Lie group $\mathrm{U}(1,n-1)\ltimes\Heis_{2n+1}$. The Lie group $L := B\ltimes \Heis_{2n+1}$, where $B$ is the Iwasawa subgroup of $\mathrm{SU}(1,n-1)$, acts simply transitively on $\bar N_\rho$ and the induced metric $g_\rho^c$ on $\bar N_\rho$ can be seen as a left-invariant metric on the solvable Lie group $L$. Hence the hypersurface $(\bar N_\rho\cong L,g_\rho^c)$ is a solvmanifold and in this paper we  study its geometry. We notice that the case $n=1$ is special, since then $L = \Heis_{3}$ is nilpotent and it is known that $g_\rho^c$ is a nilsoliton for any choice of $\rho>0$ and $c\geq 0$ (see \cite{Mil76}).\medskip

The main result of this article may be formulated as follows (see Theorems \ref{thm_heis3}, \ref{thm:c=0_soliton} and \ref{thm:c>0_not_soliton}):

\begin{theorem}\label{thm:main_theorem}
    Let $(\bar{N},g_{\bar{N}}^c)$ be the one-loop deformation of the symmetric space \eqref{eq:SU_sym_space} and let $(\bar{N}_\rho,g_\rho^c)$ be the fiber of $\rho\colon\bar{N}\to\R_{>0}$ considered as a solvmanifold. Then: 
    \begin{itemize}
        \item If $n=1$, the pair $(\bar{N}_\rho,g_\rho^c)$ is a nilsoliton for every $c\geq0$ and $\rho>0$.
        \item If $n>1$, the pair $(\bar{N}_\rho,g_\rho^c)$ is a solvsoliton for $c=0$ and $\rho>0$.
        \item If $n>1$, the pair $(\bar{N}_\rho,g_\rho^c)$ is \emph{not} a solvsoliton for $c>0$ and $\rho>0$.
    \end{itemize}
\end{theorem}

As we have said, the case $n=1$ is already known by the work of Milnor \cite{Mil76}. Moreover, in the case $c=0$ the quaternionic Kähler manifold \eqref{eq:SU_sym_space} is a symmetric space and it follows from the work of \cite{DST21} that $(\bar{N}_\rho,g_\rho^0)$ is a solvsoliton, thus our explicit computation agrees with this general result (see Remark \ref{rmk:general_sym_solv}). Our computations also show that in the case $c>0$ we do not have a solvsoliton anymore. We obverse in Remark \ref{rmk:alternative_proof} that we can also prove the third point of the theorem combining the fact that the case $c=0$ is a solvsoliton, together with the uniqueness results about solvsolitons obtained by Lauret in \cite{Lau11} and our study of the Ricci curvature of $\bar{N}_\rho$.\medskip

The results of \cite{DST21} hold, in particular, for every quaternionic Kähler symmetric space of negative scalar curvature. Then, combining \cite{DST21} and the results of Lauret \cite{Lau11}, together with an analogous study of the Ricci curvature of the one we perform here, we have enough evidence to state the following conjecture:

\begin{conjecture}
    Consider the one-loop deformation, with $c>0$, of a symmetric supergravity $c$-map space of dimension $4n>4$. Then the hypersurface orbit of its cohomogeneity one action equipped with the induced metric is not a solvsoliton.
\end{conjecture}

The only 4-dimensional quaternionic Kähler symmetric spaces of negative scalar curvature are $\HH^1$ and $\CH^2$. The former is not a supergravity $c$-map space and the latter is a special case that we already consider in this paper.\medskip

To apply the results of Lauret about the uniqueness of solvsolitons, we rely on the fact that the hypersurface orbit $L$ of the cohomogeneity one action is a solvsoliton when $c=0$ and, as we said, this is true for symmetric supergravity $c$-map spaces by \cite{DST21}. However, for a homogeneous non-symmetric quaternionic Kähler manifold of negative scalar curvature $\bar{N}$ the results of \cite{DST21} do not apply. Nevertheless, we still have the group $L$ acting with cohomogeneity one on $\bar{N}$ (although in the case $c=0$, since the quaternionic Kähler manifold is homogeneous, there is a bigger group acting simply transitively on $\bar{N}$). A natural question that arise is the following:

\begin{question}
    Consider a homogeneous non-symmetric quaternionic Kähler manifold of negative scalar curvature. Is the hypersurface orbit of the cohomogeneity one action equipped with the induced metric a solvsoliton?
\end{question}

Finally, it is shown in \cite{CS18} that the one-loop deformation of $\CH^2$ has negative sectional curvature. More precisely, the authors prove that the one-loop deformation is a complete $\frac{1}{4}$-pinched negatively curved quaternionic Kähler manifold. This corresponds to the case $n=1$ of \eqref{eq:SU_sym_space}. Some experimental evidence suggests that the deformed supergravity $c$-map space that we are considering has also negative sectional curvature. It is an interesting problem to construct complete non-homogeneous Riemannian manifolds of negative sectional curvature. We then pose the following question:

\begin{question}
    Is the one-loop deformation of \eqref{eq:SU_sym_space} a Riemannian manifold of negative sectional curvature? More generally, is the one-loop deformation of a homogeneous supergravity $c$-map space a Riemannian manifold of negative sectional curvature?
\end{question}

The paper is organized as follows. In Section \ref{sec:preliminaries_solv} we recall the definition of solvsolitons and the results relevant for the paper. In Section \ref{sec:Ricci_curvature} we obtain a general formula for the Ricci curvature of a hypersurface of an Einstein manifold (see Lemma \ref{Lem:RicNrho}) and then we apply it to our particular example to obtain the eigenvalues of the Ricci endomorphism of the hypersurface (see Proposition \ref{prop:Ricci_tensor}). In Section \ref{sec:Riemannian_solv} we first recall the Iwasawa decomposition of $\mathfrak{su}(1,n-1)$ and then we describe the Lie algebra of $L$, showing that $L$ is completely solvable and non-unimodular for $n>1$ (see Proposition \ref{prop:non-unimodular}). Then we describe the induced metric on the hypersurface as a left-invariant metric and express its Ricci endomorphism in a basis of left-invariant vector fields (see Proposition \ref{Prop:gcrho_frl}). Finally we obtain our main results in Theorems \ref{thm_heis3}, \ref{thm:c=0_soliton} and \ref{thm:c>0_not_soliton}.\medskip

This paper is part of the second named author PhD thesis, see \cite[Chapter 6]{GilGarcia24}.

\subsection*{Acknowledgements}

We thank Christoph Böhm for suggesting the problem studied in this paper. The work of VC and MR is supported by the Deutsche Forschungsgemeinschaft (DFG, German Research Foundation) under Germany’s Excellence Strategy -- EXC 2121 ``Quantum Universe'' -- 390833306 and under -- SFB-Geschäftszeichen 1624 -- Projektnummer 506632645. The work of AGG is supported by the Beijing Institute of Mathematical Sciences and Applications (BIMSA) and was previously funded by the two above-mentioned grants.


\section{Preliminaries on solvsolitons}\label{sec:preliminaries_solv}
  
Ricci solitons were introduced by Hamilton in \cite{Ham88} and they are a generalization of Einstein manifolds. More precisely, a Riemannian manifold $(M,g)$ is a \emph{Ricci soliton} if there exist a real number $\lambda\in\R$ and a vector field $V\in\Gamma(TM)$ such that $$\Ric=\lambda g+\calL_Vg,$$ where $\Ric$ denotes the Ricci curvature of $g$. Note that if $V$ is Killing then $(M,g)$ is Einstein.\medskip

We are interested in studying the existence of Ricci solitons in the homogeneous setting, more precisely Lie groups equipped with left-invariant metrics.

\begin{definition}
    Let $(G,g)$ be a simply connected Lie group equipped with a left-invariant Riemannian metric $g$, and let $\frg$ denote the Lie algebra of $G$. Then the pair $(G,g)$ is called an \emph{algebraic Ricci soliton} if there exist a real number $\lambda\in\R$ and a derivation $D\in\Der(\frg)$ such that \begin{equation}\label{eq:alg_Ricci_sol}
        \ric=\lambda\Id+D,
    \end{equation} where $\ric$ denotes the Ricci endomorphism of $g$. In particular, an algebraic Ricci soliton on a solvable (resp.\ nilpotent) Lie group is called a \emph{solvsoliton} (resp.\ \emph{nilsoliton}).
\end{definition}
  
The relationship between left-invariant Ricci solitons on simply connected Lie groups and algebraic Ricci solitons was studied by Lauret in \cite{Lau01,Lau11}. He shows that any algebraic Ricci soliton gives rise to a Ricci soliton. He proves that the converse also holds in the case of completely solvable Lie groups. Recall that $G$ is called \emph{completely solvable} if $G$ is solvable and the eigenvalues of $\ad(X)$ are real for all $X\in\frg$. Note that nilpotent Lie groups are completely solvable.
  



\begin{remark}
    It is shown in \cite[Proposition 4.6]{Lau11} that if $(G,g)$ is a solvsoliton with $\lambda\geq0$, then $\ric=0$. Thus the corresponding left-invariant Ricci soliton is Ricci-flat and hence flat by \cite{AK75}.
\end{remark}

From now on, $S$ denotes a simply connected solvable Lie group and $g$ is a left-invariant Riemannian metric on $S$. Then the pair $(S,g)$ is called a \emph{solvmanifold}. A particularly interesting case is when $(S,g)$ is Einstein. The relation between Einstein solvmanifolds and nilsolitons and solvsolitons has been deeply studied in the literature. Let us briefly recall some of these relations.\medskip
    
Let $\frs$ denote the Lie algebra of $S$. We decompose $$\frs=\fra\oplus\frn,$$ where $\frn$ is the nilradical of $\frs$, that is,  its maximal nilpotent ideal, and $\fra:=\frn^\perp$ is the orthogonal complement of $\frn$. When $\fra$ is abelian, then $(\frs=\fra\oplus\frn,g)$ is called \emph{standard}. Lauret showed in \cite{Lau10} that every Einstein solvmanifold is standard. Based on the seminal work of Heber \cite{Heb98}, Lauret proves the following relation between Einstein solvmanifolds and nilsolitons.

\begin{theorem}[{\cite[Theorem 3.7]{Lau01}}]\label{thm:nilsoliton_Einstein}
    Let $(S,g)$ be a solvmanifold and $\frs$ the Lie algebra of $S$. Consider the orthogonal decomposition $\frs=\fra\oplus\frn$, where $\frn$ is the nilradical of $\frs$. Then $(S,g)$ is Einstein if and only if $(\frn,g|_{\frn\times\frn})$ is a nilsoliton.
\end{theorem}

When $\fra$ in $\frs=\fra\oplus\frn$ is one-dimensional, Theorem \ref{thm:nilsoliton_Einstein} says that if $(S,g)$ is an Einstein solvmanifold, then $S$ admits a codimension one nilsoliton. This was extended to the unimodular case in \cite{LL14} (see also \cite{Tho24}). That is, if $(S,g)$ is an Einstein solvmanifold, then $S$ admits a unimodular codimension one closed subgroup $S_0$ which is a solvsoliton with the induced metric. Conversely, if $(S_0,g_0)$ is a solvsoliton, where $S_0$ is an unimodular solvable Lie group, then there is a semidirect product $S=\R\ltimes S_0$ which admits a left-invariant Einstein metric extending $g_0$.\medskip
    
Finally, Lauret generalized Theorem \ref{thm:nilsoliton_Einstein} in \cite{Lau11} to the case of solvsolitons and nilsolitons. In particular, he characterized the existence of solvsolitons from the existence of nilsolitons.

\begin{theorem}[{\cite[Theorem 4.8]{Lau11}}]\label{thm:nilsoliton_solvsoliton}
    Let $(S,g)$ be a solvmanifold and $\frs$ the Lie algebra of $S$. Consider the orthogonal decomposition $\frs=\fra\oplus\frn$, where $\frn$ is the nilradical of $\frs$. Then $(S,g)$ is a solvsoliton, i.e.\ $\ric=\lambda\Id+D$ for some $\lambda<0$ and $D\in\Der(\frs)$, if and only if \begin{enumerate}
        \item $(\frn,g|_{\frn\times\frn})$ is a nilsoliton.
        \item $\fra$ is abelian.
        \item $\ad(A)$ is a normal operator for all $A\in\fra$.
        \item $g(A,A)=-\tfrac{1}{\lambda}\Trace(\ad(A)^s\circ\ad(A)^s)$ for all $A\in\fra$, where $\ad(A)^s:=\tfrac{1}{2}(\ad(A)+\ad(A)^*)$ and $\ad(A)^*$ denotes the adjoint.
    \end{enumerate}
\end{theorem}
    
All the known examples of quaternionic Kähler homogeneous manifolds of negative scalar curvature admit a simply transitive action of a completely solvable Lie group \cite{Ale75,Cor96}. These are known as \emph{Alekseevsky spaces}. In fact, it was recently shown that these are precisely all quaternionic Kähler homogeneous manifolds of negative scalar curvature \cite{BL23}. Combining these results we conclude that all quaternionic Kähler homogeneous manifolds of negative scalar curvature are Einstein solvmanifolds.\medskip

Therefore, due to the results mentioned so far, given a quaternionic Kähler homogeneous manifold we can always find an unimodular hypersurface admitting a solvsoliton. However, the setting we will consider is quite different. We study the existence of algebraic Ricci solitons on hypersurface orbits coming from a cohomogeneity one group action.


\section{The Ricci curvature of a hypersurface in an Einstein manifold}\label{sec:Ricci_curvature}

Let $K$ be a smooth manifold and consider the smooth manifold $\bar N := (0,\infty)\times K$. Write $\rho\colon \bar N\to (0,\infty)$ for the canonical projection. On $\bar N$ we suppose that we have an Einstein metric $g$ of the form $$g = f(\rho)\d\rho^2 + g_\rho,$$ where $g_\rho$ is a Riemannian metric on $\bar N_{\rho}:=\{\rho\}\times K$ and $f\colon \bar N\to(0,\infty)$ is a smooth positive function depending only on $\rho$. We may think of $g_\rho$ as a one-parameter family of Riemannian metrics on $K$ depending on $\rho$. In this section we give a general formula for the Ricci tensor of the hypersurface $\bar N_\rho$. Similar computations appear in a different context in \cite{Koi81}. Of course, in the case when $\rho$ is a distance function, i.e.\ $f\equiv 1$, then the computations that follow are classical, but we nevertheless work out the formulas for general $f$, since these then readily apply to the family of quaternionic K\"ahler metrics we are interested in.
  
\subsection{Ricci tensor of \texorpdfstring{$\bar N_\rho$}{N} in general}

\begin{lemma}\label{Lem:Defh}
    The bilinear form $h\in \Gamma(T^*\bar N\otimes T^*\bar N)$ given by $h(X,Y) := g(\nabla_X\pd{}{\rho},Y)$ is symmetric, i.e.\ $g(\nabla_X\pd{}{\rho},Y) = g(\nabla_Y\pd{}{\rho},X)$.
\end{lemma}

\begin{proof}
    Consider the function $F = \int f\d\rho\colon(0,\infty)\to\R$. Then $F' = f$, so that, viewed as a function on $\bar N$, we have $\d F = f\d\rho$. It follows that $\nabla F = \pd{}{\rho}$, and $h$ is by definition the Hessian of $F$, hence symmetric.
\end{proof}
  
\begin{lemma}\label{Lem:nablarhorho}
    We have $\nabla_{\pd{}{\rho}}\pd{}{\rho} = \frac{f'}{2f}\pd{}{\rho}$.
\end{lemma}
  
  \begin{proof}
    First of all we have 
    \begin{equation}\label{Eq:nablarhorho}
      g\left(\nabla_{\pd{}{\rho}}\pd{}{\rho},\pd{}{\rho}\right) = \frac{1}{2}\pd{}{\rho}g\left(\pd{}{\rho},\pd{}{\rho}\right)   = \frac{1}{2}f'(\rho) = g\left(\frac{f'}{2f}\pd{}{\rho},\pd{}{\rho}\right). 
    \end{equation}
    To show that $\nabla_{\pd{}{\rho}}\pd{}{\rho} = \frac{f'}{2f} \pd{}{\rho}$ it suffices to show $\nabla_{\pd{}{\rho}}\pd{}{\rho}\perp \ker (\d\rho)$. To this end, we pick a vector field $X$ which is tangent to the fibers of $\rho$, i.e.\ $\d\rho(X)=0$. Then $g(X,\pd{}{\rho}) = 0$ and since $\bar N$ is a product, we may assume that $[\pd{}{\rho},X]=0$.\medskip
    
    It follows, since $\nabla$ is torsion-free:
    \begin{align*}
        g\left(\nabla_{\pd{}{\rho}}\pd{}{\rho},X\right) &= -g\left(\pd{}{\rho}, \nabla_{\pd{}{\rho}}X\right) = -g\left(\pd{}{\rho}, \nabla_X\pd{}{\rho}\right)= -\frac{1}{2}X(f) =0.
    \end{align*}
  \end{proof}

  To state the next lemma, we use the following notation. If $\alpha\in\Gamma(T^*\bar N\otimes T^*\bar N)$ is a symmetric bilinear form with associated endomorphism $A$, i.e.\ $\alpha(\cdot,\cdot) = g(A\cdot,\cdot)$ we write $$\alpha^2(\cdot,\cdot) = g(A^2\cdot,\cdot) = g(A\cdot,A\cdot) = \alpha(\cdot, A\cdot)\in\Gamma(T^*\bar N\otimes T^*\bar N).$$

    Note that if $\{E_i\}$ is an orthonormal local frame of $T\bar N$, then $AX= \sum_ig(AX,E_i)E_i = \sum_i\alpha(X,E_i)E_i$, so that 
  $$\alpha^2(X,X) = g(AX,AX) = \sum_{i,j}g(\alpha(X,E_i)E_i,\alpha(X,E_j)E_j) = \sum_i\alpha(X,E_i)^2.$$ 
  
  In particular, for the symmetric bilinear form $h$ from Lemma \ref{Lem:Defh}, we have $h(\cdot,\cdot)= g(\nabla_{\cdot}\pd{}{\rho},\cdot)$ so that $$h^2(\cdot,\cdot) = g\left(\nabla_{\cdot}\pd{}{\rho},\nabla_{\cdot}\pd{}{\rho}\right) = h\left(\cdot,\nabla_{\cdot}\pd{}{\rho}\right).$$
  
  \begin{lemma}\label{Lem:Comp}
    Let $X,Y\in\Gamma(T\bar N)$ be two vector fields tangent to $K$, i.e.\ such that $\d\rho(X) = 0 = \d\rho(Y)$, and suppose that $[X,\pd{}{\rho}] = 0 = [Y,\pd{}{\rho}]$. Then 
    \begin{enumerate}
        \item $h(X,Y) = g(\nabla_X\pd{}{\rho},Y) = -g(\nabla_XY,\pd{}{\rho}) = \frac{1}{2}\pd{}{\rho}g_\rho(X,Y)$.
        \item The second fundamental form $\mathrm{II}\colon T\bar N_\rho\times T\bar N_\rho\to T\bar N_\rho^\perp$ of $\bar N_\rho$ is given by $$\mathrm{II}(X,Y) = -\frac{1}{f}h(X,Y)\pd{}{\rho} = -\frac{1}{2f}\pd{}{\rho}g_\rho(X,Y)\pd{}{\rho}.$$
        \item $g(R_{\bar N}(X,\pd{}{\rho})\pd{}{\rho},X) = \frac{f'}{4f}\pd{}{\rho}g_\rho(X,X) - \frac{1}{2}\frac{\partial^2}{\partial\rho^2}g_\rho(X,X) + h^2(X,X)$.
    \end{enumerate}
  \end{lemma}

  \begin{proof}
  \begin{enumerate}
    
    \item Since $g(\pd{}{\rho},Y) = 0$ and $\nabla$ is metric, it follows that $g(\nabla_X\pd{}{\rho},Y) = -g(\nabla_XY,\pd{}{\rho})$. Since $g(\pd{}{\rho},X) = 0 =g(\pd{}{\rho},Y)$ and $[X,\pd{}{\rho}] = 0 = [Y,\pd{}{\rho}]$ the Koszul formula implies: 
    \begin{align*}
    2h(X,Y) &= 2g\left(\nabla_X\pd{}{\rho},Y\right) = \pd{}{\rho} g_\rho(X,Y).
    \end{align*}
    
    \item By definition we have $\mathrm{II}(X,Y) = \left(\nabla_XY\right)^\perp = g(\nabla_XY,\nu)\nu$, where $\nu = \frac{1}{\sqrt{f}}\pd{}{\rho}$ is the unit normal. Thus, using part 1.:
    $$\mathrm{II}(X,Y) = \frac{1}{f}g\left(\nabla_XY,\pd{}{\rho}\right)\pd{}{\rho} = - \frac{1}{f}h(X,Y)\pd{}{\rho} = -\frac{1}{2f}\pd{}{\rho}g_\rho(X,Y)\pd{}{\rho}.$$
    
    \item We compute using part 2.\ and Lemma \ref{Lem:nablarhorho}:
    \begin{align*}
     g\left(R_{\bar N}\left(X,\pd{}{\rho}\right)\pd{}{\rho},X\right) &= g\left(\nabla_X\nabla_{\pd{}{\rho}}\pd{}{\rho}-\nabla_{\pd{}{\rho}}\nabla_X\pd{}{\rho} -\nabla_{\left[X,{\pd{}{\rho}}\right]}\pd{}{\rho},X\right) \\
     &=g\left(\nabla_X\left(\frac{f'}{2f}\pd{}{\rho}\right),X\right)-g\left(\nabla_{\pd{}{\rho}}\nabla_X\pd{}{\rho},X\right) \\
     &=\frac{f'}{2f}g\left(\nabla_X\pd{}{\rho},X\right)-\pd{}{\rho}g\left(\nabla_X\pd{}{\rho},X\right)+ g\left(\nabla_X\pd{}{\rho},\nabla_{\pd{}{\rho}}X\right)\\
     &=\frac{f'}{4f}\pd{}{\rho}g_\rho(X,X) - \frac{1}{2}\frac{\partial^2}{\partial\rho^2}g_\rho(X,X) + h\left(X,\nabla_X\pd{}{\rho}\right)\\
     &=\frac{f'}{4f}\pd{}{\rho}g_\rho(X,X) - \frac{1}{2}\frac{\partial^2}{\partial\rho^2}g_\rho(X,X) + h^2(X,X).
    \end{align*}
    \end{enumerate}
  \end{proof}
  
  \begin{lemma}\label{lemma:Ric_Nrho_II}
    Suppose that $(\bar N,g)$ is an Einstein manifold with Einstein constant $\lambda\in\R$, i.e.\ $\mathrm{Ric}(g) = \lambda g$. Consider the hypersurface $\bar N_\rho = \{\rho\}\times K$ with induced metric $g_\rho$, unit normal field $\nu\in\Gamma(T\bar N_\rho^\perp)$ and second fundamental form $\mathrm{II}\in\Gamma(T^*\bar N_\rho\otimes T^*\bar N_\rho\otimes T\bar N_\rho^\perp)$. Then for any $p\in \bar N_\rho$ and $X\in T_p\bar N_\rho$ we have, with an orthonormal basis $\{E_i\}$ of $T_p\bar N_\rho$:
  $$\Ric_{\bar N_\rho}(X,X) = \lambda g_\rho(X,X) + g(\mathrm{II}(X,X),\mathrm{tr}(\mathrm{II})) -\sum_ig(\mathrm{II}(X,E_i),\mathrm{II}(X,E_i)) -\frac{1}{f}\Rm_{\bar N}\left(X,\pd{}{\rho},\pd{}{\rho},X\right).$$
  \end{lemma}
  
  \begin{proof}
  Let $p\in \bar N_\rho$ and $X,Y,Z,W\in T_p\bar N_\rho$. Then writing $\Rm_{\bar N}(X,Y,Z,W) = g(R_{\bar N}(X,Y)Z,W)$ and $\Rm_{\bar N_\rho}(X,Y,Z,W) = g_\rho(R_{\bar N_\rho}(X,Y)Z,W)$ we have by the Gauss equation: 
  $$\Rm_{\bar N}(X,Y,Z,W)=\Rm_{\bar N_\rho}(X,Y,Z,W)-g(\mathrm{II}(X,W),\mathrm{II}(Y,Z))+g(\mathrm{II}(X,Z),\mathrm{II}(Y,W)).$$
  Now let $\{E_i\}$ be an orthonormal basis of $T_p\bar N_\rho$, so that $\{E_i\}\cup\{\nu|_p\}$ is an orthonormal basis of $T_p\bar N$, with $\nu = \frac{1}{\sqrt{f}}\pd{}{\rho}$. Then we have 
  \begin{align*}
    \Ric_{\bar N}(X,X) &= \mathrm{tr}(V\mapsto R_{\bar N}(V,X)X)\\
    &= \sum_i \Rm_{\bar N}(X,E_i,E_i,X) + \Rm_{\bar N}(X,\nu,\nu,X)\\
    &= \sum_i\big(\Rm_{\bar N_\rho}(X,E_i,E_i,X) -g(\mathrm{II}(X,X),\mathrm{II}(E_i,E_i)) + g(\mathrm{II}(X,E_i),\mathrm{II}(E_i,X))\big)\\
    &\quad + \Rm_{\bar N}(X,\nu,\nu,X)\\
    &= \Ric_{\bar N_\rho}(X,X) - g(\mathrm{II}(X,X),\mathrm{tr}(\mathrm{II})) + \sum_i g(\mathrm{II}(X,E_i),\mathrm{II}(X,E_i)) +\frac{1}{f}\Rm_{\bar N}\left(X,\pd{}{\rho},\pd{}{\rho},X\right).
  \end{align*}

  Now $\Ric_{\bar N} = \lambda g$ and therefore we find the formula of the statement.
  \end{proof}
  
  Write $h_\rho\in\Gamma(T^*\bar N_\rho\otimes T^*\bar N_\rho)$ for the restriction of $h$ to the hypersurface $\bar{N}_\rho$. We refine the formula obtained in Lemma \ref{lemma:Ric_Nrho_II} as follows.

\begin{lemma}\label{Lem:RicNrho}
    Suppose that $(\bar N,g)$ is an Einstein manifold with Einstein constant $\lambda\in\R$, i.e.\ $\mathrm{Ric}(g) = \lambda g$. Consider the hypersurface $\bar N_\rho = \{\rho\}\times K$ with induced metric $g_\rho$, unit normal field $\nu\in\Gamma(T\bar N_\rho^\perp)$ and second fundamental form $\mathrm{II}\in\Gamma(T^*\bar N_\rho\otimes T^*\bar N_\rho\otimes T\bar N_\rho^\perp)$. Then we have
  $$\Ric_{\bar N_\rho} = \lambda g_\rho + \left(\frac{1}{4f}\mathrm{tr}\left(\pd{}{\rho}g_\rho\right)- \frac{f'}{4f^2}\right)\pd{}{\rho}g_\rho -\frac{2}{f}h_\rho^2+\frac{1}{2f}\frac{\partial^2}{\partial\rho^2}g_\rho.$$
  \end{lemma}
  
  \begin{proof}
  Let $p\in\bar N_\rho$ and $X\in T_p\bar N_\rho$. By Lemma \ref{Lem:Comp} part 2.\ we have $$g(\mathrm{II}(X,X),\mathrm{tr}(\mathrm{II}))=g\left(-\frac{1}{f}h_\rho(X,X)\pd{}{\rho},-\frac{1}{f}\mathrm{tr}(h_\rho)\pd{}{\rho}\right)=\frac{1}{f}h_\rho(X,X)\mathrm{tr}(h_\rho)=\frac{1}{4f}\pd{}{\rho}g_\rho(X,X)\mathrm{tr}\left(\pd{}{\rho}g_\rho\right)$$ and similarly $$\sum_ig(\mathrm{II}(X,E_i),\mathrm{II}(X,E_i))=\frac{1}{f}\sum_ih_\rho(X,E_i)h_\rho(X,E_i)=\frac{1}{f}h_\rho^2(X,X).$$

  By Lemma \ref{Lem:Comp} part 3.\ the curvature term is given by $$\Rm_{\bar N}\left(X,\pd{}{\rho},\pd{}{\rho},X\right)=\frac{f'}{4f}\pd{}{\rho}g_\rho(X,X)-\frac{1}{2}\frac{\partial^2}{\partial\rho^2}g_\rho(X,X)+h_\rho^2(X,X).$$

  Putting all these formulas together we obtain: \begin{align*}
      \mathsf{A}(X)&:=g(\mathrm{II}(X,X),\mathrm{tr}(\mathrm{II}))-\sum_ig(\mathrm{II}(X,E_i),\mathrm{II}(X,E_i))-\frac{1}{f}\Rm_{\bar N}\left(X,\pd{}{\rho},\pd{}{\rho},X\right)\\
      &\phantom{:}=\frac{1}{4f}\pd{}{\rho}g_\rho(X,X)\mathrm{tr}\left(\pd{}{\rho}g_\rho\right)-\frac{1}{f}h_\rho^2(X,X)-\frac{1}{f}\left(\frac{f'}{4f}\pd{}{\rho}g_\rho(X,X)-\frac{1}{2}\frac{\partial^2}{\partial\rho^2}g_\rho(X,X)+h_\rho^2(X,X)\right)\\
      &\phantom{:}=\frac{1}{4f}\pd{}{\rho}g_\rho(X,X)\mathrm{tr}\left(\pd{}{\rho}g_\rho\right)-\frac{2}{f}h_\rho^2(X,X)-\frac{f'}{4f^2}\pd{}{\rho}g_\rho(X,X)+\frac{1}{2f}\frac{\partial^2}{\partial\rho^2}g_\rho(X,X)\\
      &\phantom{:}=\left(\frac{1}{4f}\mathrm{tr}\left(\pd{}{\rho}g_\rho\right)-\frac{f'}{4f^2}\right)\pd{}{\rho}g_\rho(X,X) -\frac{2}{f}h_\rho^2(X,X)+\frac{1}{2f}\frac{\partial^2}{\partial\rho^2}g_\rho(X,X)
  \end{align*} and therefore $\Ric_{\bar N_\rho}(X,X)=\lambda g_\rho(X,X)+\mathsf{A}(X)$ gives us the formula from the statement.
\end{proof}

\subsection{Application to one-loop deformed hypermultiplet manifolds}

Let us consider the one-loop deformation (with $c>0$) of the non-compact quaternionic Kähler symmetric space $$\frac{\mathrm{SU}(n,2)}{\mathrm{S}(\mathrm{U}(n)\times\mathrm{U}(2))}.$$

We identify the underlying manifold $\bar N$ as 
$$\bar N = (0,\infty)\times \left(B_1(0)\times\R\times \C^n\right),$$ with $B_1(0)\subset \C^{n-1}$ the open unit ball. On $\bar N$ we have the global coordinate system $$(\rho,X^a,\tilde\phi,w^0,w^a)\in (0,\infty)\times (\C^{n-1}\times\R\times \C^n),$$ where $a=1,\ldots,n-1$ and $\norm{X}^2=\sum_a \abs{X^a}^2<1$. If $n=1$ we adopt the convention that the family $(X^a,w^a)_{1}^0$ is empty, so that for $\bar N = (0,\infty)\times \R\times\C$ the global coordinate system is just $(\rho, \tilde\phi, w^0)$. With the convention $\sum_{a=1}^0 = 0$,  the one-loop deformed metric is then for any $n\in\mathbb N$ explicitly given by
$$g_{\bar{N}}^c=\frac{1}{4\rho^2}\frac{\rho+2c}{\rho+c}\d\rho^2+g_\rho^c,$$ where \begin{align*}
    g_\rho^c&=\frac{\rho+c}{\rho}\frac{1}{1-\norm{X}^2}\left(\sum_{a=1}^{n-1}\abs{\d X^a}^2+\frac{1}{1-\norm{X}^2}\left|\sum_{a=1}^{n-1}\bar{X}^a\d X^a\right|^2\right)\\
    &\quad+\frac{1}{4\rho^2}\frac{\rho+c}{\rho+2c}\left(\d\tilde{\phi}-4\Im\left(\bar{w}^0\d w^0-\sum_{a=1}^{n-1}\bar{w}^a\d w^a\right)+\frac{2c}{1-\norm{X}^2}\Im\left(\sum_{a=1}^{n-1}\bar{X}^a\d X^a\right)\right)^2\\
    &\quad-\frac{2}{\rho}\left(\d w^0\d\bar{w}^0-\sum_{a=1}^{n-1}\d w^a\d\bar{w}^a\right)+\frac{\rho+c}{\rho^2}\frac{4}{1-\norm{X}^2}\left|\d w^0+\sum_{a=1}^{n-1}X^a\d w^a\right|^2.
\end{align*}

The metric $g_{\bar{N}}^c$ fits into the framework of the previous section with $$K = B_1(0)\times\R\times\C^n\quad\text{and}\quad f(\rho)= \frac{1}{4\rho^2}\frac{\rho+2c}{\rho+c}.$$

The metric $g^c_{\bar N}$ is quaternionic K\"ahler with reduced scalar curvature $\nu = \frac{\mathrm{scal}}{4n(n+2)} = -2$ (see \cite[Page 89]{CDS17}). In particular, it is Einstein with Einstein constant $$\lambda =\frac{\mathrm{scal}}{4n} = -2(n+2).$$

We are interested in computing the Ricci tensor of the hypersurface $\bar{N}_\rho$ with metric $g_\rho^c$. For that we want to apply Lemma \ref{Lem:RicNrho} and therefore we need to compute $\pd{}{\rho}g_\rho^c$, $\frac{\partial^2}{\partial\rho^2}g_\rho^c$ and the bilinear form $h_\rho^2$.\medskip

The level sets $\bar N_\rho$ of $\rho$ are homogeneous, thus it suffices to perform all computations at the particular point $$p_\rho = (\rho, 0)\in \bar N_{\rho} = \{\rho\}\times K$$ obtained by fixing $\rho$ and putting all other coordinates to zero. Moreover, we are only differentiating with respect to the variable $\rho$, so we may first put the other coordinates to zero, view $g_\rho^c|_{p_\rho}$ as a depending only on $\rho$, and then differentiate. We may thus work with the metric $g_\rho^c$ in the following simplified form:
\begin{align*}
  g_\rho^c&=\frac{\rho+c}{4\rho}\left(\sum_{a=1}^{n-1}(\d b^a)^2+(\d t^a)^2\right)+\frac{1}{4\rho^2}\frac{\rho+c}{\rho+2c}\d\tilde{\phi}^2\\
  &\quad+\frac{1}{2\rho}\frac{\rho+2c}{\rho}((\d \tilde\zeta_0)^2+(\d\zeta^0)^2)+\frac{1}{2\rho}\left(\sum_{a=1}^{n-1}(\d \tilde\zeta_a)^2+(\d\zeta^a)^2\right).
\end{align*}

Here we have expressed the metric $g_\rho^c$ in the real coordinates $(\rho,b^a,t^a,\tilde\phi,\tilde\zeta_0,\zeta^0,\tilde\zeta_a,\zeta^a)$, where $$X^a = \frac{1}{2}(b^a+it^a),\quad w^0 = \frac{1}{2}(\tilde\zeta_0+i\zeta^0),\quad w^a = \frac{1}{2}(\tilde\zeta_a-i\zeta^a).$$

Now we compute the derivative of the metric $g_\rho^c$: 
 \begin{align*}
		\pd{}{\rho}g_\rho^c&=-\frac{c}{4\rho^2}\left(\sum_{a=1}^{n-1}(\d b^a)^2+(\d t^a)^2\right)-\frac{1}{4\rho^3}\frac{2\rho^2+5c\rho+4c^2}{(\rho+2c)^2}\d\tilde{\phi}^2\\
		&\quad-\frac{1}{2\rho^2}\frac{\rho+4c}{\rho}((\d \tilde\zeta_0)^2+(\d\zeta^0)^2)-\frac{1}{2\rho^2}\left(\sum_{a=1}^{n-1}(\d\tilde\zeta_a)^2+(\d\zeta^a)^2\right)\\
		&=-\frac{1}{\rho}\left(h_1(\rho)\frac{\rho+c}{4\rho}\left(\sum_{a=1}^{n-1}(\d b^a)^2+(\d t^a)^2\right)+h_2(\rho)\frac{1}{4\rho^2}\frac{\rho+c}{\rho+2c}\d\tilde{\phi}^2\right)\\
		&\quad-\frac{1}{\rho}\left(h_3(\rho)\frac{1}{2\rho}\frac{\rho+2c}{\rho}((\d\tilde\zeta_0)^2+(\d\zeta^0)^2)+\frac{1}{2\rho}\left(\sum_{a=1}^{n-1}(\d\tilde\zeta_a)^2+(\d\zeta^a)^2\right)\right),
	\end{align*} where 
 $$h_1(\rho):=\frac{c}{\rho+c}>0,\quad h_2(\rho):=\frac{2\rho^2+5c\rho+4c^2}{(\rho+c)(\rho+2c)}>0,\quad h_3(\rho):=\frac{\rho+4c}{\rho+2c}>0.$$
 Note that the Gram matrix of $g_\rho^c$ is diagonal in these coordinates:
 $$g_\rho^c =\mathrm{diag}\left(\frac{\rho+c}{4\rho}\mathbbm{1}_{2n-2} , \frac{\rho +c}{4\rho^2(\rho+2c)} , \frac{\rho+2c}{2\rho^2}\mathbbm{1}_2 , \frac{1}{2\rho}\mathbbm{1}_{2n-2}\right),$$
 where we denote by $\mathbbm{1}_k$ the $k\times k$ identity matrix. Sometimes we will also write $\mathbb{O}_k$ for the $k\times k$ zero matrix. If $n=1$, then $2n-2=0$ and we adopt the convention to interpret the Gram matrix of $g_\rho^c$ as 
  $$g_\rho^c =\mathrm{diag}\left(\frac{\rho +c}{4\rho^2(\rho+2c)} , \frac{\rho+2c}{2\rho^2}\mathbbm{1}_2\right),$$
  which is consistent with our conventions of choosing coordinates on $\bar N$ explained at the beginning of this chapter. We apply analogous conventions to the various other Gram matrices that appear below and henceforth will not explicitly distinguish between the cases $n=1 $ and $n>1$.
 It follows that the Gram matrix of the bilinear form $h_\rho(\cdot,\cdot) = g_\rho^c(\nabla_{\cdot}\pd{}{\rho},\cdot)= \frac{1}{2}\pd{}{\rho}g_\rho^c$ is
 $$H_\rho = \frac{1}{2}\pd{g_\rho^c}{\rho} = -\frac{1}{2\rho}\mathrm{diag}\left(
   h_1(\rho)\frac{\rho+c}{4\rho}\mathbbm{1}_{2n-2} , h_2(\rho)\frac{\rho +c}{4\rho^2(\rho+2c)}, h_3(\rho)\frac{\rho+2c}{2\rho^2}\mathbbm{1}_2, \frac{1}{2\rho}\mathbbm{1}_{2n-2}  
 \right) = g_\rho^c A_\rho,$$
 where $A_{\rho}$ is the diagonal matrix
 \begin{equation}\label{Eq:Arho}
 A_\rho =-\frac{1}{2\rho}\mathrm{diag}\big(h_1(\rho)\mathbbm{1}_{2n-2}, h_2(\rho), h_3(\rho)\mathbbm{1}_2,\mathbbm{1}_{2n-2}\big)
 \end{equation}
 corresponding to the endomorphism $\nabla\pd{}{\rho}$.
From this computation and Lemma \ref{Lem:Comp} we deduce the following.

\begin{proposition}
The eigenvalues of the shape operator $S_\rho^c$ of the hypersurface $(\bar N_\rho,g_\rho^c)\subset (\bar N,g^c_{\bar N})$ with respect to the unit normal field $\frac{1}{\sqrt{f}}\pd{}{\rho}$ are given by $$\begin{aligned}
    \sigma_1 &= \frac{c}{\rho+c}\sqrt{\frac{\rho+c}{\rho+2c}},\\
    \sigma_2 &=  \frac{2\rho^2+5c\rho+4c^2}{(\rho+2c)(\rho+c)}\sqrt{\frac{\rho+c}{\rho+2c}},\end{aligned}\qquad\begin{aligned}
    \sigma_3 &=\frac{\rho+4c}{\rho+2c}\sqrt{\frac{\rho+c}{\rho+2c}},\\
    \sigma_4 &= \sqrt{\frac{\rho+c}{\rho+2c}},
\end{aligned}$$ where the multiplicities of $\sigma_1$ and $\sigma_4$ are $2n-2$, the multiplicity of $\sigma_2$ is $1$ and the multiplicity of $\sigma_3$ is $2$. In particular, if $c>0$, then $(\bar N_\rho,g_\rho^c)$ is strictly convex. Furthermore, the mean curvature of $(\bar N_\rho,g_\rho^c)$ is $$\mathrm{tr}(S_\rho^c) = \frac{(2n+2)\rho^2+(8n+7)c\rho+(8n+4)c^2}{(\rho+c)(\rho+2c)}\sqrt{\frac{\rho+c}{\rho+2c}}.$$
\end{proposition}
     
 \begin{proof}
     By part (3) of Lemma \ref{Lem:Comp} we know that the second fundamental form $\mathrm{II}_\rho$ of the hypersurface $(\bar N_\rho,g_\rho^c)\subset (\bar N,g^c_{\bar N})$ evaluates on tangent vectors $X,Y$ to 
     $$\mathrm{II}_\rho(X,Y) = -\frac{1}{f}h_\rho(X,Y)\pd{}{\rho} =  g_\rho^c(S_\rho^c(X),Y)\frac{1}{\sqrt f}\pd{}{\rho},$$
     with the shape operator $S_\rho^c$. At the point $p_\rho$ we see from the above discussion that 
     $$g_\rho^c(S_\rho^c(X),Y) = -\frac{1}{\sqrt f}h_\rho(X,Y) = g_\rho^c\left(\left(-\frac{1}{\sqrt{f}}A_\rho\right)X,Y\right),$$
     i.e.\ from \eqref{Eq:Arho} we get $$S_\rho^c = -\frac{1}{\sqrt{f}}A_\rho = \frac{1}{2\rho\sqrt{f}}\mathrm{diag}\big(h_1(\rho)\mathbbm{1}_{2n-2}, h_2(\rho), h_3(\rho)\mathbbm{1}_2,\mathbbm{1}_{2n-2}\big).$$
     The eigenvalues are explicitly given by
     \begin{align*}
         \sigma_1 &= \frac{h_1(\rho)}{2\rho \sqrt{f}} = \frac{c}{2\rho(\rho+c)\frac{1}{2\rho}\sqrt{\frac{\rho+2c}{\rho+c}}} = \frac{c}{\rho+c}\sqrt{\frac{\rho+c}{\rho+2c}},\\
         \sigma_2 &= \frac{h_2(\rho)}{2\rho \sqrt{f}} = \frac{2\rho^2+5c\rho+4c^2}{(\rho+c)(\rho+2c)2\rho\frac{1}{2\rho}\sqrt{\frac{\rho+2c}{\rho+c}}} = \frac{2\rho^2+5c\rho+4c^2}{(\rho+2c)(\rho+c)}\sqrt{\frac{\rho+c}{\rho+2c}},\\
         \sigma_3 &= \frac{h_3(\rho)}{2\rho \sqrt{f}} = \frac{\rho+4c}{(\rho+2c)2\rho\frac{1}{2\rho}\sqrt{\frac{\rho+2c}{\rho+c}}} =\frac{\rho+4c}{\rho+2c}\sqrt{\frac{\rho+c}{\rho+2c}},\\
         \sigma_4 &= \frac{1}{2\rho\sqrt{f}} = \sqrt{\frac{\rho+c}{\rho+2c}}.
     \end{align*}
     The mean curvature \begin{align*}
        \Trace(S_\rho^c)&=(2n-2)\sigma_1+\sigma_2+2\sigma_3+(2n-2)\sigma_4\\
        &=\frac{(2n+2)\rho^2+(8n+7)c\rho+(8n+4)c^2}{(\rho+c)(\rho+2c)}\sqrt{\frac{\rho+c}{\rho+2c}}
    \end{align*} is obtained from a straightforward computation.
\end{proof}
 
 We deduce that the trace of the bilinear form $\pd{}{\rho}g_\rho^c = 2h_\rho$ is given by 
 $$\mathrm{tr}\left(\pd{}{\rho}g_\rho^c\right) = 2\mathrm{tr}(A_\rho) = -\frac{1}{\rho}\big((2n-2)h_1(\rho)+h_2(\rho)+2h_3(\rho)+2n-2\big).$$
 Moreover, it follows that $(\nabla\pd{}{\rho})^2$ is represented by 
 $$A_\rho^2 =\frac{1}{4\rho^2}\mathrm{diag}\left(h_1^2(\rho)\mathbbm{1}_{2n-2}, h_2^2(\rho), h_3^2(\rho)\mathbbm{1}_2, \mathbbm{1}_{2n-2}  
 \right).$$
 
 The Gram matrix of $h_\rho^2$, the bilinear form corresponding to the endomorphism $(\nabla\pd{}{\rho})^2$, is therefore
 $$H_\rho^2 = A_\rho^2g_\rho^c = \frac{1}{4\rho^2}\mathrm{diag}\left(h_1^2(\rho)\frac{\rho+c}{4\rho}\mathbbm{1}_{2n-2}, h_2^2(\rho)\frac{\rho +c}{4\rho^2(\rho+2c)}, h_3^2(\rho)\frac{\rho+2c}{2\rho^2}\mathbbm{1}_2, \frac{1}{2\rho}\mathbbm{1}_{2n-2}\right).$$
 It remains to compute the second derivative $\frac{\partial^2}{\partial\rho^2}g_\rho^c$, the Gram matrix of which is 
 \begin{align*}
   \frac{\partial^2}{\partial\rho^2}g_\rho^c &= \pd{}{\rho}\left(\pd{}{\rho}g_\rho^c\right) = 2\pd{}{\rho}H_\rho = 2\pd{}{\rho}\left(A_\rho g_\rho^c\right)\\
   &=2\left(\pd{A_\rho}{\rho}g_\rho^c + A_{\rho}\pd{g_\rho^c}{\rho}\right) =2\left(\pd{A_\rho}{\rho}g_\rho^c + 2A_{\rho}H_\rho\right) = 2\left(\pd{A_\rho}{\rho}+2A_\rho^2\right)g_\rho^c.
 \end{align*}
 All terms have already been computed, except $\pd{A_\rho}{\rho}$, which is
 \begin{align*}\pd{A_\rho}{\rho} &= \frac{1}{2\rho^2}\mathrm{diag}\big((h_1(\rho)-\rho h_1'(\rho))\mathbbm{1}_{2n-2}, h_2(\rho)-\rho h_2'(\rho),(h_3(\rho)-\rho h_3'(\rho))\mathbbm{1}_2,\mathbbm{1}_{2n-2}\big).
 \end{align*}
 We find
 \begin{align*}
   \frac{\partial^2}{\partial\rho^2}g_\rho^c &= \left(\pd{A_\rho}{\rho}+2A_\rho^2\right)2g_\rho^c \\
   &= \frac{1}{\rho^2}\mathrm{diag}\left((h_1^2+h_1-\rho h_1')\mathbbm{1}_{2n-2}, (h_2^2+h_2-\rho h_2'), (h_3^2+h_3-\rho h_3')\mathbbm{1}_2, 2\mathbbm{1}_{2n-2}\right)g_\rho^c.
 \end{align*}
 
 Since $\Ric(g_{\bar N}^c) = -2(n+2)g_{\bar N}^c$, we must take $\lambda = -2(n+2)$ in the formula of Lemma \ref{Lem:RicNrho}. All other quantities in that formula are now computed, so putting everything together, we find: 

\begin{proposition}\label{prop:Ricci_tensor}
  Let $n\in\mathbb N$ and $c\geq 0$. The Ricci tensor of $(\bar N_\rho,g_\rho^c)$ at a point $p_\rho=(\rho,0)$ is given by 
  $$\Ric_{\bar N_\rho} = -2(n+2)g_\rho^c -2n\rho\pd{}{\rho}g_\rho^c +\frac{1}{f}g_\rho^c\left(\pd{A_\rho}{\rho}\cdot,\cdot\right).$$
  
  If $n=1$, then in the global real coordinates $(\tilde\phi,\tilde\zeta_0,\zeta^0)$ the Ricci endomorphism is represented by the diagonal matrix
  $$\ric_{\bar N_\rho} = \mathrm{diag}\big(r_2,r_3\mathbbm{1}_2\big).$$
  
  If $n>1$, then in the global real coordinates $(b^a,t^a,\tilde\phi,\tilde\zeta_0,\zeta^0,\tilde\zeta_a,\zeta^a)$, $a=1,\dots,n-1$, the Ricci endomorphism is represented by the diagonal matrix
  $$\ric_{\bar N_\rho} = \mathrm{diag}\big(r_1\mathbbm{1}_{2n-2},r_2,r_3\mathbbm{1}_2,r_4\mathbbm{1}_{2n-2}\big).$$
  The principal Ricci curvatures are, for any $n\in\mathbb N$ and $c\geq 0$, given by
  \begin{align*}
    r_1 &=  \frac{-2(n+2)\rho^2-4(n+2)c\rho -6c^2}{(\rho+c)(\rho+2c)},\\
    r_2 &= \frac{2n\rho^4+(12n-8)c\rho^3+(28n-26)c^2\rho^2 +32(n -1)c^3\rho +16(n-1)c^4}{(\rho+c)(\rho+2c)^3},\\
    r_3 &=  2\frac{-\rho^3+(2n-3)c\rho^2+(8n-8)c^2\rho+(8n-8)c^3}{(\rho+2c)^3},\\
    r_4 &=  -\frac{2(\rho+3c)}{\rho+2c}.
  \end{align*}
\end{proposition}

\begin{proof}
  The general formula of Lemma \ref{Lem:RicNrho} becomes with $\lambda= -2(n+2)$:
  \begin{align*}\Ric_{\bar N_\rho} &= -2(n+2)g_\rho^c + \left(\frac{1}{4f}\mathrm{tr}\left(\pd{}{\rho}g_\rho^c\right)- \frac{f'}{4f^2}\right)\pd{}{\rho}g_\rho^c -\frac{2}{f}h_\rho^2+\frac{1}{2f}\frac{\partial^2}{\partial\rho^2}g_\rho^c\\
  &= -2(n+2)g_\rho^c + \frac{1}{4f} \left(\mathrm{tr}\left(\pd{}{\rho}g_\rho^c\right)- \frac{f'}{f}\right)\pd{}{\rho}g_\rho^c -\frac{2}{f}h_\rho^2+\frac{1}{2f}\frac{\partial^2}{\partial\rho^2}g_\rho^c,
  \end{align*}
  where $f = \frac{\rho+2c}{4\rho^2(\rho+c)}$. We can therefore compute 
  \begin{align*}
  \mathrm{tr}\left(\pd{}{\rho}g_\rho^c\right)- \frac{f'}{f}&= -\frac{1}{\rho}\left((2n-2)h_1(\rho)+h_2(\rho)+2h_3(\rho)+2n-2\right) + \frac{4\rho^2(\rho+c)}{\rho+2c}\frac{2\rho^2+7c\rho+4c^2}{4\rho^3(\rho+c)^2}\\
  &=-\frac{1}{\rho}\left((2n-2)\frac{\rho+2c}{\rho+c}+\frac{2\rho^2+5c\rho+4c^2}{(\rho+c)(\rho+2c)}+2\frac{\rho+4c}{\rho+2c}\right) + \frac{2\rho^2+7c\rho+4c^2}{\rho(\rho+c)(\rho+2c)}\\
  &= -2n\frac{(\rho+2c)}{\rho(\rho+c)} = -8n\rho f.
  \end{align*}
  Thus, we get 
  $$\frac{1}{4f}\left(\mathrm{tr}\left(\pd{}{\rho}g_\rho^c\right)-\frac{f'}{f}\right) = \frac{1}{4f}(-8n\rho f)= -2n\rho$$
  and the formula for $\Ric_{\bar N_\rho}$ simplifies to 
  \begin{align*}\Ric_{\bar N_\rho} &= -2(n+2)g_\rho^c -2n\rho\pd{}{\rho}g_\rho^c -\frac{2}{f}h_\rho^2+\frac{1}{2f}\frac{\partial^2}{\partial\rho^2}g_\rho^c.
  \end{align*}
  
  We next simplify the last two terms. In terms of Gram matrices we have $$-2h^2_\rho+\frac{1}{2}\frac{\partial^2}{\partial \rho^2}g_\rho^c = -2A_\rho^2g_\rho^c+\frac{1}{2}\left(\pd{A_\rho}{\rho}+2A_\rho^2\right)2g_\rho^c=\pd{A_\rho}{\rho}g_\rho^c.$$
  
  Thus, we find 
  \begin{align*}\Ric_{\bar N_\rho} &= -2(n+2)g_\rho^c -2n\rho\pd{}{\rho}g_\rho^c +\frac{1}{f}g_\rho^c\left(\pd{A_\rho}{\rho}\cdot,\cdot\right).
  \end{align*}
  We therefore obtain for the Gram matrix of the bilinear form $\Ric_{\bar N_\rho}$, using $\frac{1}{2\rho^2f} = \frac{2(\rho+c)}{\rho+2c}$:
  \begin{align*}
    \Ric_{\bar N_\rho}&= -2(n+2)g_\rho^c -2n\rho 2H_\rho+\frac{1}{f}\pd{A_\rho}{\rho}g_\rho^c\\
    &= \left(-2(n+2)\mathbbm{1}_{4n-1} - 4n\rho A_\rho+\frac{1}{f}\pd{A_\rho}{\rho}\right)g_\rho^c\\
    &= \mathrm{diag}\big(r_1\mathbbm{1}_{2n-2},r_2,r_3\mathbbm{1}_2,r_4\mathbbm{1}_{2n-2}\big)g_\rho^c,
  \end{align*}
  with 
  \begin{align*}
  r_1 &= -2(n+2)+2nh_1- \frac{2(\rho+c)(\rho h_1'-h_1)}{\rho+2c}\\
  &= \frac{-2(n+2)\rho^2-4(n+2)c\rho -6c^2}{(\rho+c)(\rho+2c)},\\
  r_2 &= -2(n+2)+2nh_2- \frac{2(\rho+c)(\rho h_2'-h_2)}{\rho+2c}\\ 
  &=\frac{2n\rho^4+(12n-8)c\rho^3+(28n-26)c^2\rho^2 +32(n -1)c^3\rho +16(n-1)c^4}{(\rho+c)(\rho+2c)^3},\\
  r_3 &= -2(n+2)+2nh_3- \frac{2(\rho+c)(\rho h_3'-h_3)}{\rho+2c}\\
  &= 2\frac{-\rho^3+(2n-3)c\rho^2+(8n-8)c^2\rho+(8n-8)c^3}{(\rho+2c)^3},\\
  r_4 &= -2(n+2)+2n+ \frac{2(\rho+c)}{\rho+2c} = -\frac{2(\rho+3c)}{\rho+2c}.
  \end{align*}
\end{proof}

\begin{remark}
We comment on the nature of the principal Ricci curvatures computed in Proposition \ref{prop:Ricci_tensor}. 
\begin{enumerate}
    \item If $n=1$, then $g_\rho^c$ is a left-invariant metric on the three-dimensional Heisenberg group (see Section \ref{sec:Riemannian_solv}). For any $c\geq 0$, the Ricci endomorphism has just two eigenvalues, namely $$r_2 = \frac{2\rho^4 + 4c\rho^3+2c^2\rho^2}{(\rho+c)(\rho+2c)^3}=\frac{2\rho^2(\rho+c)}{(\rho+2c)^3}=-r_3.$$ Thus, $\ric_{\bar N_\rho}$ in this case is of the form $r_2\diag(1,-1,-1)$, and in fact the Ricci endomorphism of any invariant metric on the three-dimensional Heisenberg group  may be put in this form (see \cite{Mil76}).
    \item If $c=0$ and $n>1$, the principal Ricci curvatures simplify to 
    \begin{equation}
        r_1 = -2(n+2),\quad r_2 = 2n,\quad r_3=r_4 = -2.
    \end{equation}
    Note in particular that $r_3=r_4$ in this case and the spectrum of $\ric_{\bar N_\rho}$ consists only of three distinct eigenvalues. It follows that $\ric_{\bar N_\rho}$ restricts to a multiple of the identity on the subspace of $T_{p_\rho}\bar N_\rho$ spanned by $\pd{}{\zeta^0},\pd{}{\tilde\zeta_0},\dots,\pd{}{\zeta^{n-1}},\pd{}{\tilde\zeta_{n-1}}$. 
    \item In the case $n>1$, $c>0$ the eigenvalues $r_1,r_2,r_3,r_4$ are distinct, and, in contrast to the $c=0$ situation, the vectors $\pd{}{\zeta^0},\pd{}{\tilde\zeta_0}$ and $\pd{}{\zeta^1},\pd{}{\tilde\zeta_1},\dots,\pd{}{\zeta^{n-1}},\pd{}{\tilde\zeta_{n-1}}$ now span distinct eigenspaces of $\ric_{\bar N_\rho}$. 
\end{enumerate}
\end{remark}

From these remarks we deduce the following important consequence.

\begin{corollary}\label{cor:non-homothetic}
    Let $n>1$. The metrics $g_\rho=g_\rho^0$ and $g_\rho^c$, $c>0$, are not homothetic.
\end{corollary}




\section{The level sets \texorpdfstring{$(\bar{N}_\rho, g_\rho^c)$}{(N,g)} as Riemannian solvmanifolds}\label{sec:Riemannian_solv}

It was shown in \cite{CRT21} that the level sets $\bar{N}_\rho$ are the orbits for an isometric, $c$-dependent action of the simply connected Lie group with Lie algebra $\mathfrak{u}(1,n-1)\ltimes\mathfrak{heis}_{2n+1}$. The Lie algebra of the stabilizer of the point $p_\rho = (\rho,0)$ with respect to this action is a subgroup of $\mathfrak{u}(1,n-1)\ltimes\heis_{2n+1}$ isomorphic to $\mathfrak{u}(1)\times\mathfrak{u}(n-1)$. In this section we shall see that the subgroup $L$ with Lie algebra $\frl = \frb\ltimes \heis_{2n+1}$, where $\frb\subset\mathfrak{su}(1,n-1)$ is the Iwasawa subalgebra, acts simply transitively and isometrically on $\bar N_\rho$. We may thus regard $g_\rho^c$ as a left-invariant metric on the Lie group $L$, and in this section we shall determine the subalgebra $\frl$ and the inner product on $\frl$ corresponding to $g_\rho^c$ explicitly.

\subsection{The Iwasawa decomposition of \texorpdfstring{$\mathfrak{su}(1,n-1)$}{su(1,n-1)}}

Decomposing $\C^n = \C e_0\oplus \C^{n-1}$ with an orthonormal basis $e_0,e_1,\ldots,e_{n-1}$ of $\C^n$ with respect to the pseudo-Hermitian inner product $h$ of signature $(1,n-1)$ such that $h(e_0,e_0) =-1$, we may write $$\mathfrak{u}(1,n-1) = \R C \oplus \mathfrak{su}(1,n-1),$$ where
$$C = \begin{pmatrix}
i & 0\\
0 & i\mathbbm{1}_{n-1}
\end{pmatrix}\quad\text{and}\quad\mathfrak{su}(1,n-1) =\left\{\begin{pmatrix}
-\mathrm{tr}(A) & \bar v^\top \\ v & A \end{pmatrix} \, \Big| \, v\in\C^{n-1}, A\in\mathfrak{u}(n-1) \right\}.$$

We view $\mathfrak{u}(1,n-1)\subset\mathfrak{gl}(n,\C)$ as the fixed-point set of the anti-linear involutive Lie algebra automorphism $$\sigma\colon \mathfrak{gl}(n,\C)\to\mathfrak{gl}(n,\C),\quad \sigma(A):=A^\sigma:= -I\bar{A}^\top I,\quad\text{where }I = \begin{pmatrix}
-1 & 0\\
0 & \mathbbm{1}_{n-1}
\end{pmatrix}.$$

Note that $(AB)^\sigma = -B^\sigma A^\sigma$. Given $A\in\mathfrak{gl}(n,\C)$ we then write 
$$\Re(A) =\frac{1}{2}(A+A^\sigma)\quad\text{and}\quad\Im(A) =\frac{1}{2i}(A-A^\sigma).$$

For $a=1,\ldots,n-1$ we write further $$U_a = \begin{pmatrix}
0 & e_a^\top\\
0 & 0
\end{pmatrix}\quad\text{and}\quad U_a^\sigma:=\sigma(U_a)=\begin{pmatrix}
  0&0\\
  e_a&0
\end{pmatrix}.$$

We observe that $$[U_a,U_b]=0,\quad[U_a^\sigma,U_b^\sigma]=0,\quad[U_a,U_b^\sigma]=\begin{pmatrix}
  \delta_{ab}&0\\
  0&-e_be_a^\top
\end{pmatrix}.$$

Then $$\{C, \Re(U_a),\Im(U_a), \Re([U_a,U_b^\sigma]),\Im([U_a,U_b^\sigma])\mid a,b=1,\ldots,n-1\}$$ is a basis for $\mathfrak{u}(1,n-1)$.\medskip

The real vector space underlying the Lie algebra $\mathfrak{heis}_{2n+1}$ is given by $\C^n\oplus \R$, where we identify $\C^n$ with $\R^{2n}$ and write $Z$ for the generator of the center $\R$.\medskip

We now fix $\rho\in(0,\infty)$, take as basepoint $p_\rho := (\rho,0)\in \bar{N}_\rho$ and consider the infinitesimal action of the Lie algebra $\mathfrak{su}(1,n-1)\ltimes\mathfrak{heis}_{2n+1}$ on the hypersurface $\bar{N}_\rho$. The Lie algebra of the stabilizer $\mathfrak{g}'$ of $p_\rho$, i.e.\ the kernel of the map $\mathfrak{u}(1,n-1)\ltimes\mathfrak{heis}_{2n+1}\to T_{p_\rho}\bar{N}_\rho$ given by evaluating the Killing fields at $p_\rho$, was computed in \cite[Lemma 3.5]{CRT21}: $$\mathfrak{g}' = \mathrm{span}_\R\{C+2cZ, \Re([U_a,U_b^\sigma]),\Im([U_a,U_b^\sigma]) + 2c\delta_{ab}Z\mid a,b=1,\ldots,n-1\},$$
which is isomorphic to $\mathfrak{u}(1)\times\mathfrak{u}(n-1)$ and has trivial intersection with $\mathfrak{heis}_{2n+1}$.\medskip

We briefly review the Iwasawa decomposition of $\mathfrak{su}(1,n-1)$. We may choose the following Cartan decomposition 
$$\mathfrak{u}(1,n-1) = \mathfrak{k}\oplus \mathfrak{p},$$
where \begin{align*}
  \mathfrak{k} &= \R C \oplus \left\{\begin{pmatrix}
-\mathrm{tr}(A) & 0\\ 0 & A \end{pmatrix} \, \Big| \,  A\in\mathfrak{u}(n-1) \right\} \cong \mathfrak{u}(1)\oplus\mathfrak{u}(n-1),\\
\mathfrak{p} &= \left\{\begin{pmatrix}
0 & \bar v^\top \\ v & 0 \end{pmatrix} \, \Big| \, v\in\C^{n-1}\right\}.
\end{align*}

Define for $a\in\{1,\ldots,n-1\}$:
$$B_a := (1+\delta_{1a})U_a-[U_a,U_1^\sigma] = \begin{pmatrix}
    -\delta_{1a} & (1+\delta_{1a})e_a^\top\\
    0 & e_1e_a^\top
\end{pmatrix}$$
and \begin{align*}
		B_a^R&:=\Re(B_a)=\frac{1}{2}\begin{pmatrix}
			0&(1+\delta_{1a})e_a^\top\\
			(1+\delta_{1a})e_a&e_1e_a^\top-e_ae_1^\top
		\end{pmatrix},\\
		B_a^I&:=\Im(B_a)=\frac{1}{2i}\begin{pmatrix}
			-2\delta_{1a}&(1+\delta_{1a})e_a^\top\\
			-(1+\delta_{1a})e_a&e_1e_a^\top+e_ae_1^\top
		\end{pmatrix}.
    \end{align*}

Then $$B_1^R=\begin{pmatrix}
		0&e_1^\top\\
		e_1&0
	\end{pmatrix},\quad B_1^I=\begin{pmatrix}
		i&-i&0\\
		i&-i&0\\
		0&0&0
	\end{pmatrix}$$ and for $a>1$ we have $$B_a^R=\frac{1}{2}\begin{pmatrix}
		0&0&\tilde{e}_{a-1}^\top\\
		0&0&\tilde{e}_{a-1}^\top\\
		\tilde{e}_{a-1}&-\tilde{e}_{a-1}&0
	\end{pmatrix},\quad B_a^I=\frac{1}{2}\begin{pmatrix}
		0&0&-i\tilde{e}_{a-1}^\top\\
		0&0&-i\tilde{e}_{a-1}^\top\\
		i\tilde{e}_{a-1}&-i\tilde{e}_{a-1}&0
	\end{pmatrix},$$ where $\tilde{e}_{a-1}\in\C^{n-2}$ is such that $e_a^\top=\begin{pmatrix}
		0&\tilde{e}_{a-1}^\top
    \end{pmatrix}\in\C^{n-1}$.\medskip
    
A maximal abelian subalgebra of $\mathfrak{p}$ is given by 
$$\mathfrak{a} := \left\{\begin{pmatrix}
0 & ae_1^\top \\ ae_1 & 0 \end{pmatrix} \, \Big| \, a\in\R \right\} = \mathrm{span}_{\R}\{B_1^R\}.$$

The positive eigenvalues of $\ad(B_1^R) = [B_1^R,\cdot] \in\mathrm{End}(\mathfrak{u}(1,n-1))$ are 2 and 1 with eigenspaces given by 
\begin{align*}
\mathfrak{g}_{2} &= \left\{\begin{pmatrix}
ia & -iae_1^\top \\iae_1 & -iae_1e_1^\top \end{pmatrix} \, \Big| \, a\in\R \right\} = \mathrm{span}_{\R}\{B_1^I\},\\
\mathfrak{g}_{1} &= \left\{\begin{pmatrix}
0 & 0 & \bar z^\top \\
0 & 0 & \bar z^\top\\
z & -z & 0 \end{pmatrix} \, \Big| \, z\in\C^{n-2} \right\}= \mathrm{span}_{\R}\{B_a^R,B_a^I\mid a=2\ldots, n-1\}.
\end{align*}
With 
$$\mathfrak n = \mathfrak g_1\oplus\mathfrak g_2 = \left\{\begin{pmatrix}
ia & -ia & \bar z^\top \\
ia & -ia & \bar z^\top\\
z & - z & 0 \end{pmatrix} \, \Big| \, z\in\C^{n-2}, a\in\R \right\}$$
we get $\mathfrak{u}(1,n-1)=\mathfrak{k}\oplus\fra\oplus\frn$ and the Iwasawa decomposition 
$$\mathfrak{su}(1,n-1) = \mathfrak{u}(n-1)\oplus\mathfrak a \oplus\mathfrak{n}.$$
We observe the following bracket relations for $i=1,2$:
$$[\mathfrak{a},\mathfrak{a}] = 0,\quad [\mathfrak{a},\mathfrak{g}_i]\subset \mathfrak{g}_i,\quad [\mathfrak{g}_1,\mathfrak{g}_1]\subset \mathfrak{g}_2,\quad [\mathfrak{g}_1,\mathfrak{g}_2] = 0 = [\mathfrak{g}_2,\mathfrak{g}_2],$$
which imply
$$[\mathfrak{b},\mathfrak{b}]\subset\mathfrak{n},\quad [\mathfrak{n},\mathfrak{n}]\subset\mathfrak{g}_2, \quad [\mathfrak{n},[\mathfrak{n},\mathfrak{n}]]=0.$$ 

\begin{lemma}\label{Lem:frnheis}
    The Lie algebra $\frn$ is isomorphic to $\heis_{2n-3}$ and the basis vectors $B_1^I,B_a^R,B_a^I$, where $ a=2,\ldots,n-1$, satisfy the following non-trivial bracket relations (all other brackets are zero): 
    $$[B_a^R,B_a^I] = \frac{1}{2}B_1^I.$$
\end{lemma}

\begin{proof}
    Since $[\frg_1,\frg_2] = [\frg_2,\frg_2] = \{0\}$ we see that $\frg_2$ is contained in the center of $\frn$. 
    Let $ z_1,z_2\in\C^{n-2}$. We can compute directly
    $$\left[\begin{pmatrix}
        0 & 0 & \bar z_1^\top\\
        0 & 0 & \bar z_1^\top\\
        z_1 & -z_1 & 0
    \end{pmatrix},\begin{pmatrix}
        0 & 0 & \bar z_2^\top\\
        0 & 0 & \bar z_2^\top\\
        z_2 & -z_2 & 0
    \end{pmatrix}\right] = \begin{pmatrix}
        2i\Im(\bar z_1^\top z_2) & -2i\Im(\bar z_1^\top z_2) & 0\\
        2i\Im(\bar z_1^\top z_2) & -2i\Im(\bar z_1^\top z_2) & 0\\
        0 & 0 & 0
    \end{pmatrix} = 2\Im(\bar z_1^\top z_2)B_1^I.$$
    It follows that $\frg_2 = \mathrm{span}_\R\{B_1^I\}$ is the center of $\frn$ and we note that $(z_1,z_2)\mapsto 2\Im(\bar z_1^\top z_2)$ is a nonzero multiple of the standard symplectic (K\"ahler) form on $\C^{n-2}$. Thus, $\frn\cong\heis_{2n-3}$.\medskip
    
    Since $B_a^R$ corresponds to choosing $z=\frac{1}{2}\tilde{e}_{a-1}$ and $B_a^I$ corresponds to choosing $z=\frac{i}{2}\tilde{e}_{a-1}$ if $a>1$ we find
    $$[B_a^R,B_b^R]= [B_a^I,B_b^I] = \frac{1}{2}\Im(\tilde{e}_{a-1}^\top\tilde{e}_{b-1})B_1^I = 0, \quad [B_a^R,B_b^I] = \frac{1}{2}\Im(i\tilde{e}_{a-1}^\top\tilde{e}_{b-1})B_1^I = \frac{1}{2}\delta_{ab}B_1^I.$$
\end{proof}

Consider the $(2n-2)$-dimensional real solvable subalgebra 
$$\mathfrak{b} = \mathfrak a\oplus\mathfrak n = \left\{\begin{pmatrix}
ia & b-ia & \bar z^\top \\
b+ia & -ia & \bar z^\top\\
z & - z & 0 \end{pmatrix} \, \Big| \, z\in\C^{n-2}, a,b\in\R \right\}\subset\mathfrak{su}(1,n-1).$$

A basis for $\mathfrak{b}$ is given by
$$\{B_a^R,B_a^I\mid a=1,\ldots,n-1\}.$$
We have by construction
$$\fra = \mathrm{span}_{\R}\{B_1^R\},\quad \frg_2 = \mathrm{span}_{\R}\{B_1^I\},\quad \frg_1 = \mathrm{span}_{\R}\{B_a^R, B_a^I\mid a=2,\ldots,n-1\}.$$

The subalgebra $\frn\cong\heis_{2n-3}$ is an ideal in $\frb$ and we have computed the brackets of the basis vectors of $\frn$ in Lemma \ref{Lem:frnheis}. The next lemma is then clear from the definition of $\frg_1$, $\frg_2$.

\begin{lemma}\label{Lem:bracketsfrb}
    We have 
    $$[B_1^R,B_1^I] = 2B_1^I,\quad [B_1^R,B_a^R] = B_a^R,\quad [B_1^R,B_a^I] = B_a^I$$
    for any $a\in\{2,\ldots,n-1\}$. 
\end{lemma}

Note that $\mathfrak{b}\cap\mathfrak{g}'=\{0\}$. It follows that the action of the $(4n-1)$-dimensional real solvable Lie algebra
$$\mathfrak{l}:= \mathfrak{b}\ltimes\mathfrak{heis}_{2n+1}$$
on $\bar{N}_\rho$ is free. Here $\frb$ acts on $\heis_{2n+1}\cong \C^n\oplus\R$ by the standard representation of $\mathfrak{u}(1,n-1)$ on $\C^n$.\medskip

\subsection{The Lie algebra \texorpdfstring{$\frl$}{l}}

Now let $e_0,f_0,e_a,f_a$, $a=1,\ldots,n-1$, be the standard basis of $\R^{2n}$. We define the one-dimensional central extension $\heis_{2n+1}$ of $\R^{2n}$ by setting
\begin{equation*}\label{eq:commutators}
	[e_k,e_l]=0, \quad [f_k,f_l]=0,\quad 
	[e_k,f_l]=\bigg(\delta_{k0}\delta_{l0}-\sum_{a=1}^{n-1}\delta_{ka}\delta_{la}\bigg)Z
\end{equation*} 
for every $k,l=0,1,\ldots,n-1$, where $Z$ denotes the generator of the center. Complexifying and extending the Lie bracket complex-bilinearly, we obtain $\heis_{2n+1}^\C$.\medskip

Set $E_k\coloneqq e_k-if_k$.

\begin{lemma}\label{lemma:brackets_semidirect}
    The brackets between the elements of the basis $\{B_a^R,B_a^I\mid a=1,\ldots,n-1\}\subset\frb$ and of the  complex basis $\{E_k,\bar E_k, Z\mid k=0,\ldots,n-1\}$ are given as follows:
    \begin{align*}
    [B_1^R,E_k]&= \overline{[B_1^R,\bar E_k]} = -\delta_{k0}E_1 -\delta_{k1}E_0,\\
    [B_a^R,E_k]&= \overline{[B_a^R,\bar E_k]} = -\frac{1}{2}(\delta_{k0}+\delta_{k1})E_a-\frac{1}{2}\delta_{ka}(E_0-E_1),\\
    [B_1^I,E_k]&= \overline{[B_1^I,\bar E_k]} = -i(\delta_{k0}+\delta_{k1})(E_0-E_1),\\
    [B_a^I,E_k]&= \overline{[B_a^I,\bar E_k]} = \frac{i}{2}(\delta_{k0}+\delta_{k1})E_a -\frac{i}{2}\delta_{ka}(E_0-E_1).
    \end{align*}
\end{lemma}

\begin{proof}
    By definition of the semidirect product structure, the bracket of $\mathfrak{gl}(n,\C)\ltimes\heis_{2n+1}^\C$ evaluated on $A\in\mathfrak{u}(1,n-1)$ and $v\in\R^{2n}$ is just $[A,v] = -A^\top v$ (where $A^\top $ is identified with a real $2n\times2n$-matrix), while we have $[A,Z] = 0$. Using this prescription, the following brackets were computed in \cite[Proposition 3.4]{CRT21}: $$\begin{aligned}
			[U_a,E_k]&=-\delta_{k0}E_a,\\
			[U_a,\bar{E}_k]&=-\delta_{ka}\bar{E}_0,
		\end{aligned}\qquad\begin{aligned}
			[U_a^\sigma,E_k]&=-\delta_{ka}E_0,\\
			[U_a^\sigma,\bar{E}_k]&=-\delta_{k0}\bar{E}_a.
		\end{aligned}$$
        
    From these identities, we may deduce the brackets between elements of $\frb$ and $\heis_{2n+1}$ using the Jacobi identity: $$\begin{aligned}
			[B_1,E_k]&=-(2\delta_{k0}+\delta_{k1})E_1+\delta_{k0}E_0,\\
			[B_1,\bar E_k]&=-(2\delta_{k1}+\delta_{k0})\bar E_0+\delta_{k1}\bar E_1,\\
			[B_a,E_k]&=-(\delta_{k0}+\delta_{k1})E_a,\\
			[B_a,\bar E_k]&=-\delta_{ka}(\bar E_0-\bar E_1),
		\end{aligned}\qquad\begin{aligned}
			[B_1^\sigma,E_k]&=-(2\delta_{k1}+\delta_{k0})E_0+\delta_{k1}E_1,\\
			[B_1^\sigma,\bar E_k]&=-(2\delta_{k0}+\delta_{k1})\bar{E}_1+\delta_{k0}\bar{E}_0,\\
			[B_a^\sigma,E_k]&=-\delta_{ka}(E_0-E_1),\\
			[B_a^\sigma,\bar E_k]&=-(\delta_{k0}+\delta_{k1})\bar{E}_a.
		\end{aligned}$$
        
    Using these relations and that $A^R=\Re(A)=\frac{1}{2}(A+A^\sigma)$ and $A^I=\Im(A)=\frac{1}{2i}(A-A^\sigma)$ for any $A\in\mathfrak{gl}(n,\C)$ we get the claimed result.
\end{proof}

For $n\in\N$, we work in the ordered basis 
\begin{equation}\label{Eq:DefcalB}
    \calB_n:=\begin{cases} \left(e_0,f_0,Z\right) & n=1\\ 
    \left(B_1^R,B_1^I,\ldots, B_{n-1}^R,B_{n-1}^I,e_0,f_0,e_1,f_1,\ldots,e_{n-1},f_{n-1},Z\right) & n>1
    \end{cases}
\end{equation}
  
\begin{proposition}\label{prop:non-unimodular}
    For $n>1$, the Lie algebra $\frl$ is completely solvable and non-unimodular.
\end{proposition}

\begin{proof}
    The adjoint operator of $B_1^R$ with respect to the basis $\mathcal B_n$ defined in \eqref{Eq:DefcalB} is given by $$\ad(B_1^R)=\mathrm{diag}\big(0,2,\mathbbm{1}_{2n-4},\mathcal{V}_4,\mathbb{O}_{2n-4},0\big),$$ where $$\mathcal{V}_4:=\begin{pmatrix}
    0&0&-1&0\\
    0&0&0&-1\\
    -1&0&0&0\\
    0&-1&0&0
  \end{pmatrix}.$$

  Then $\mathrm{tr}(\ad(B_1^R))=2n-2\neq0$, which implies that $\frl$ is non-unimodular. Moreover, the eigenvalues of $\ad(B_1^R)$ are $2,1,0,-1$ and for any $X\in\calB_n\setminus\{B_1^R\}$ the operator $\ad(X)$ is nilpotent, so its only eigenvalue is zero. Hence $\frl$ is completely solvable.
\end{proof}

\begin{remark}\label{remark:n=1_nilpotent}
    Note that the case $n=1$ corresponds to $\frl=\heis_3$, so $\frl$ is nilpotent, hence unimodular and completely solvable.
\end{remark}

\subsection{The metric \texorpdfstring{$g_\rho^c$}{g} as a left-invariant metric on \texorpdfstring{$L$}{L}}

Our goal here is to write the induced metric $g_\rho^c$ as a left-invariant metric on the Lie group $L$, and for this it is enough to compute the induced inner product on $\mathfrak{l}$, under the identification $T_eL\cong T_{p_\rho}\bar{N}_\rho$ given by the infinitesimal action of $\mathfrak{u}(1,n-1)\ltimes\mathfrak{heis}_{2n+1}$ on $\bar{N}_\rho$. The complexification of the infinitesimal action is given by the anti-homomorphism $\alpha^\C\colon\mathfrak{gl}(n,\C)\ltimes\mathfrak{heis}_{2n+1}^{\C}\to\Gamma(T\bar N)^{\C}$, where the vector fields generating such action are given by (see \cite[Proposition 3.1]{CRT21}): \begin{align*}
		Y_a&\coloneqq\alpha^\C(U_a)= \pd{}{\bar X^a}-X^a\sum_{b=1}^{n-1}X^b \pd{}{X^b}
		-w^0\pd{}{w^a}-\bar w^a\pd{}{\bar w^0} + ic X^a \pd{}{\tilde\phi},\\
		\bar Y_a&=\alpha^\C(U_a^\sigma),\\
        [Y_a,\bar Y_b]&=-\alpha^\C([U_a,U_b^\sigma])\\
        &=\delta_{ab}\Bigg(\!\sum_j \!\bigg( \!X^j\pd{}{X^j} \!- \! \bar X^j\pd{}{\bar X^j} \! \bigg) \! + \! w^0\pd{}{w^0} \! - \! \bar w^0\pd{}{\bar w^0} \! - \! 2ic\pd{}{\tilde\phi}\Bigg)\\
        &\quad+X^a\pd{}{X^b} - \bar X^b\pd{}{\bar X^a} + \bar w^a \pd{}{\bar w^b} - w^b\pd{}{w^a}.
\end{align*}

The action of the complexified Heisenberg Lie algebra $\mathfrak{heis}_{2n+1}^\C$ is then generated by the following vector fields (see \cite[Proposition 3.3]{CRT21}):
$$\pd{}{\tilde\phi} = \alpha^\C(Z),\quad V_a = \alpha^\C(E_a)= \frac{1}{\sqrt{2}}\left(\pd{}{w^a} - 2i\bar w^a\pd{}{\tilde\phi}\right), \quad V_0 = \alpha^\C(E_0)= \frac{1}{\sqrt 2}\left(\pd{}{w^0} +2i\bar w^0 \pd{}{\tilde\phi}\right),$$ where $a=1,\ldots,n-1$.\medskip

Define the map $$\alpha_\rho^\C\colon \mathfrak{l}^\C\to T_{p_\rho}\bar N_\rho^\C, \quad l\mapsto \alpha^\C(l)|_{p_\rho},$$
i.e.\ evaluation of the corresponding complex Killing field at $p_\rho$. Then the formulas above allow us to explicitly evaluate $\alpha_\rho^\C$ on the basis vectors of $\mathfrak{l}$:
\begin{align*}
\alpha_\rho^\C(B_a) &= \alpha_\rho^\C((1+\delta_{1a})U_a-[U_a,U_1^\sigma]) = (1+\delta_{1a})Y_a|_{p_\rho} + [Y_a,\bar Y_1]|_{p_\rho} = (1+\delta_{1a})\frac{\partial}{\partial \bar X^a} -2ic\delta_{1a}\frac{\partial}{\partial\tilde\phi},\\
\alpha_\rho^\C(E_k) &= \alpha_\rho^\C(e_k-if_k) = V_k|_{p_\rho} = \frac{1}{\sqrt{2}}\frac{\partial}{\partial w^k},\\
\alpha_\rho^\C(Z) &= \frac{\partial}{\partial\tilde\phi}.
\end{align*}

Let us consider again the real coordinates $X^a=\frac{1}{2}(b^a+it^a)$, $w^0=\frac{1}{2}(\tilde\zeta_0+i\zeta^0)$ and $w^a=\frac{1}{2}(\tilde\zeta_a-i\zeta^a)$. Then we find for $a=2,\dots,n-1$, $j=1,\dots,n-1$:
\begin{itemize}
    \item $\alpha_{\rho}(B_1^R)=\Re\big((2Y_1+[Y_1,\bar Y_1])|_{p_\rho}\big)=2\pd{}{b^1}$,
    \item $\alpha_{\rho}(B_1^I)=\Im\big((2Y_1+[Y_1,\bar Y_1])|_{p_\rho}\big)=2\pd{}{t^1}-2c\pd{}{\tilde\phi}$,
    \item $\alpha_{\rho}(B_a^R)=\Re\big((Y_a+[Y_a,\bar Y_1])|_{p_\rho}\big)=\pd{}{b^a}$,
    \item $\alpha_{\rho}(B_a^I)=\Im\big((Y_a+[Y_a,\bar Y_1])|_{p_\rho}\big)=\pd{}{t^a}$,
    \item $\alpha_{\rho}(e_0)=\Re\big(V_0|_{p_\rho}\big)=\frac{1}{\sqrt{2}}\pd{}{\tilde\zeta_0}$,
    \item $\alpha_{\rho}(f_0)=-\Im\big(V_0|_{p_\rho}\big)=\frac{1}{\sqrt{2}}\pd{}{\zeta^0}$,
    \item $\alpha_{\rho}(e_j)=\Re\big(V_j|_{p_\rho}\big)=\frac{1}{\sqrt{2}}\pd{}{\tilde\zeta_j}$,
    \item $\alpha_{\rho}(f_j)=-\Im\big(V_j|_{p_\rho}\big)=-\frac{1}{\sqrt{2}}\pd{}{\zeta^j}$,
    \item $\alpha_{\rho}(Z)=\pd{}{\tilde{\phi}}$,
\end{itemize}
where the vector fields on the right-hand side are evaluated at the point $p_\rho$.\medskip

At the point $p_\rho$, the metric expressed in these real coordinates is the following: \begin{align*}
 g_\rho^c&=\frac{1}{4\rho^2}\frac{\rho+c}{\rho+2c}\d\tilde{\phi}^2+\frac{\rho+c}{4\rho}\sum_{a=1}^{n-1}\big((\d b^a)^2+(\d t^a)^2\big)\\
 &\quad+\frac{\rho+2c}{2\rho^2}\big((\d\tilde{\zeta}_0)^2+(\d\zeta^0)^2\big)+\frac{1}{2\rho}\sum_{a=1}^{n-1}\big((\d\tilde{\zeta}_a)^2+(\d\zeta^a)^2\big).
\end{align*}

Let us denote by $E_{i,j}$ the matrix with $1$ in the $(i,j)$-position and zero elsewhere.

\begin{proposition}\label{Prop:gcrho_frl}
    Let $n\in\N$ and consider the basis $\calB_n$ of $\frl$ defined in \eqref{Eq:DefcalB}. Then: \begin{enumerate}
        \item If $n>1$, the Gram matrix of the inner product corresponding to $g_\rho^c$ is given by $$g_\rho^c=\mathrm{diag}\big(\tfrac{\rho+c}{\rho},\tfrac{(\rho+c)^3}{\rho^2(\rho+2c)},\tfrac{\rho+c}{4\rho}\mathbbm{1}_{2n-4},\mathcal{G}_4,\tfrac{1}{4\rho}\mathbbm{1}_{2n-4},\tfrac{1}{4\rho^2}\tfrac{\rho+c}{\rho+2c}\big)-\tfrac{c}{2\rho^2}\tfrac{\rho+c}{\rho+2c}\big(E_{2,4n-1}+E_{4n-1,2}\big),$$ where $$\mathcal{G}_4:=\begin{pmatrix}
            \tfrac{\rho+2c}{4\rho^2}\mathbbm{1}_2&0\\
            0&\tfrac{1}{4\rho}\mathbbm{1}_2
        \end{pmatrix}.$$
        \item If $n>1$, the Ricci endomorphism $\ric_\rho^c$ is in the above basis represented by the matrix: 
        $$\ric_\rho^c= \mathrm{diag}\left(r_1\mathbbm{1}_{2n-2}, r_3\mathbbm{1}_2, r_4\mathbbm{1}_{2n-2},r_2\right) + 2c(r_1-r_2)E_{4n-1,2}.$$
        \item If $n=1$, the Gram matrix of $g_\rho^c$ and the matrix of the Ricci endomorphism are
          $$g_\rho^c=\left(\begin{array}{rrr}
                \frac{2 \, c + \rho}{4 \, \rho^{2}} & 0 & 0 \\
                0 & \frac{2 \, c + \rho}{4 \, \rho^{2}} & 0 \\
                0 & 0 & \frac{c + \rho}{4 \, {\left(2 \, c + \rho\right)} \rho^{2}}
        \end{array}\right)\quad\text{and}\quad\ric_\rho^c=\frac{2\rho^2 \, {\left(c + \rho\right)}}{(2\, c+\rho)^3}\left(\begin{array}{rrr}
        -1 & 0 & 0 \\
        0 & -1 & 0 \\
        0 & 0 & 1
        \end{array}\right).$$
    \end{enumerate}
\end{proposition}

\begin{proof}
   The expressions for $g_\rho^c$ follow directly by plugging the explicit tangent vectors into $g_\rho^c$.\medskip
   
   We have computed the Ricci endomorphism with respect to the basis of $T_{p_\rho}\bar N_\rho$ given by coordinate vector fields in Proposition \ref{prop:Ricci_tensor}. It is then straightforward to evaluate $\ric_\rho^c$ on the basis $\calB_n$.
\end{proof}

In general, the Ricci endomorphism of $(\frl,g_\rho^c)$ is given by (see e.g.\ \cite[Equation 21]{Lau11}): $$\ric_\rho^c=R-\tfrac{1}{2}B-\ad(H)^s,$$ where $$\ad(H)^s=\tfrac{1}{2}(\ad(H)+\ad(H)^*)$$
is the symmetric part of $\ad(H)$ and $H\in\fra$ is the \emph{mean curvature vector}, characterized by $$g_\rho^c(H,A)=\Trace(\ad(A))$$ for all $A\in\fra$. The term $B\in\mathrm{End}(\frl)$ denotes the symmetric endomorphism defined by the Killing form of $\frl$ relative to $g_\rho^c$, that is $$g_\rho^c(BX,X)=\Trace(\ad(X)\circ\ad(X))$$ for all $X\in\frl$. The symmetric endomorphism $R$ is defined by $$g_\rho^c(RX,X)=-\tfrac{1}{2}\sum g_\rho^c([X,L_i],L_j)^2+\tfrac{1}{4}\sum g_\rho^c([L_i,L_j],X)^2,$$ where $\{L_i\}$ is an orthonormal basis of $(\frl,g_\rho^c)$. In our situation we already know $\ric^c_\rho$, so we can compute $R = \ric^c_\rho +\frac{1}{2}B+\ad(H)^s$ if we are able to determine $B$ and $H$.

\begin{lemma}\label{lemma:curvature_term_R}
    With respect to our explicit choice of basis $\mathcal B_n$ for $\frl$ we have: 
    \begin{enumerate}
        \item For $n>1$, the mean curvature vector is $$H=\frac{\Trace(\ad(B_1^R))}{g_\rho^c(B_1^R,B_1^R)}B_1^R=(2n-2)\frac{\rho}{\rho+c}B_1^R$$ and we find 
        \begin{align*}\ad(H)^s &= (2n-2)\mathrm{diag}\big(0,\tfrac{2\rho^2+4c\rho+c^2}{(\rho+c)(\rho+2c)},\tfrac{\rho}{\rho+c}\mathbbm{1}_{2n-4},\mathcal{S}_4,\mathbb{O}_{2n-4},-\tfrac{c^2}{(\rho+c)(\rho+2c)}\big)\\
            &\quad+(2n-2)\big(-\tfrac{c}{2(\rho+c)(\rho+2c)}E_{2,4n-1}+\tfrac{2c(\rho+c)}{\rho+2c}E_{4n-1,2}\big),
        \end{align*} where $$\mathcal{S}_4:=\begin{pmatrix}
            0&0&-\tfrac{\rho}{\rho+2c}&0\\
            0&0&0&-\tfrac{\rho}{\rho+2c}\\
            -1&0&0&0\\
            0&-1&0&0
        \end{pmatrix}.$$
        \item If $n>1$, the symmetric endomorphism associated with the Killing form is $B=(2n+4)\frac{\rho}{\rho+c}E_{1,1}.$
        \item If $n=1$, then $B=0=H$.
    \end{enumerate}
\end{lemma}

\begin{proof}
    Both assertions are proved by direct calculations. The details are as follows.\medskip
    
    (1) Since in our case $\fra=\mathrm{span}_{\R}\{B_1^R\}$, we have that $H$ is a multiple of $B_1^R$. In particular $$H=\frac{\Trace(\ad(B_1^R))}{g_\rho^c(B_1^R,B_1^R)}B_1^R=(2n-2)\frac{\rho}{\rho+c}B_1^R.$$
    
    We have computed $\ad(B_1^R)$ and $g_\rho^c$, in the basis $\calB_n$, in Proposition \ref{prop:non-unimodular} and \ref{Prop:gcrho_frl}, respectively. Hence we can compute the adjoint operator of $\ad(B_1^R)$: \begin{align*}
        \ad(B_1^R)^*&=\mathrm{diag}\big(0,\tfrac{2(\rho+c)^2}{\rho(\rho+2c)},\mathbbm{1}_{2n-4},\mathcal{V}_4^*,\mathbb{O}_{2n-4},-\tfrac{2c^2}{\rho(\rho+2c)}\big)\\
        &\quad-\tfrac{c}{\rho(\rho+2c)}E_{2,4n-1}+\tfrac{4c(\rho+c)^2}{\rho(\rho+2c)}E_{4n-1,2},
        \end{align*} where $$\mathcal{V}_4^*:=\begin{pmatrix}
            0&0&-\tfrac{\rho}{\rho+2c}&0\\
            0&0&0&-\tfrac{\rho}{\rho+2c}\\
            -\tfrac{\rho+2c}{\rho}&0&0&0\\
            0&-\tfrac{\rho+2c}{\rho}&0&0
        \end{pmatrix}.$$
        
    We then have \begin{align*}
        \ad(B_1^R)^s&=\mathrm{diag}\big(0,\tfrac{2\rho^2+4c\rho+c^2}{\rho(\rho+2c)},\mathbbm{1}_{2n-4},\tilde{\mathcal{S}}_4,\mathbb{O}_{2n-4},-\tfrac{c^2}{\rho(\rho+2c)}\big)\\
        &\quad-\tfrac{c}{2\rho(\rho+2c)}E_{2,4n-1}+\tfrac{2c(\rho+c)^2}{\rho(\rho+2c)}E_{4n-1,2},
        \end{align*} where $$\tilde{\mathcal{S}}_4:=\begin{pmatrix}
            0&0&-\tfrac{\rho+c}{\rho+2c}&0\\
            0&0&0&-\tfrac{\rho+c}{\rho+2c}\\
            -\tfrac{\rho+c}{\rho}&0&0&0\\
            0&-\tfrac{\rho+c}{\rho}&0&0
        \end{pmatrix}.$$
        
    Then $$\ad(H)^s=(2n-2)\frac{\rho}{\rho+c}\ad(B_1^R)^{s}.$$
    
    (2) Let us consider the Killing form $\beta(X,Y)=\Trace(\ad(X)\circ\ad(Y))$. Then $\beta(B_1^R,B_1^R)=2n+4$ and $\beta(X,Y) = 0$ for all $Y\in \mathcal B_n$ and $X\in\mathcal B_n\setminus\{B_1^R\}$. This implies that $$B=(g_\rho^c)^{-1}\beta=(2n+4)\frac{\rho}{\rho+c}E_{1,1}.$$
    
    (3) This is clear, since $\fra =0$ and $\frl = \heis_{3}$ is nilpotent.
\end{proof}

\subsection{Existence of solvsolitons on \texorpdfstring{$\bar N_\rho$}{N}}

Now that we have determined the structure of the Lie algebra $\frl$ and have obtained explicit formulas for the metric $g^c_\rho$ and its Ricci endomorphism $\ric^c_\rho$, we can determine whether $(\frl,g_\rho^c)$ is a solvsoliton or not.

\begin{lemma}\label{Lem:DerL}
    Consider the Lie algebra $\frl = \frb\ltimes\heis_{2n+1}$ and write $\heis_{2n+1} = \C^n\oplus \R Z$.  Consider $\delta\in\mathrm{End}(\frl)$ given by $\delta|_{\frb} = 0$, $\delta Z = 2Z$ and $\delta V = V$ for all $V\in \C^{n}$. Then $\delta$ is a derivation of $\frl$.
\end{lemma}

\begin{proof}
    First note that the endomorphism $\delta$ acts as zero on $\frb=\fra\oplus\frg_2\oplus\frg_1$. For the Heisenberg Lie algebra $\heis_{2n+1}$ we have $[\heis_{2n+1},\heis_{2n+1}] \subset \R Z$ and $\delta Z = 2Z$. This implies that for all $V+tZ,W+uZ\in\heis_{2n+1}=\C^n\oplus\R Z$ we have 
    $$\delta[V+tZ,W+uZ]=2[V+tZ,W+uZ] = 2[V,W].$$ 
    On the other hand \begin{align*}
        [\delta(V+tZ),W+uZ]+[V+tZ,\delta(W+uZ)]&=[V+2tZ,W+uZ] + [V+tZ,W+2uZ]\\
        &= [V,W]+[V,W] =  2[V,W].
    \end{align*}
    
    Using that $\delta B=0$ and $[B,V+tZ]\in\C^n\oplus\{0\}\subset \heis_{2n+1}$ for all $B\in\frb$ and all $V+tZ\in\heis_{2n+1}$ and a straightforward computation we conclude that $\delta \in\Der(\frl)$.
\end{proof}

As we have pointed out in Remark \ref{remark:n=1_nilpotent}, the case $n=1$ is special, so we consider it separately.

\begin{theorem}\label{thm_heis3}
    Let $n=1$. Then the pair $(\frl,g_\rho^c)$ is a nilsoliton for any $c\geq 0$ and $\rho>0$.
\end{theorem}

\begin{proof}
    In the case $n=1$ we have that $\frb$ is zero-dimensional and then $\frl=\heis_3$. 
    By Proposition \ref{Prop:gcrho_frl} the Ricci endomorphism with respect to the basis $\calB_1$ is 
    $$\ric_\rho^c=K\left(\begin{array}{rrr}
        -1 & 0 & 0 \\
        0 & -1 & 0 \\
        0 & 0 & 1
    \end{array}\right),$$ with $K:=\frac{2\rho^2 \, {\left(c + \rho\right)}}{(2\, c+\rho)^3}$. If we take $\lambda=-3K$ and $$D=2K\delta =2K\left(\begin{array}{rrr}
        1 & 0 & 0 \\
        0 & 1 & 0 \\
        0 & 0 & 2
        \end{array}\right)\in\Der(\heis_3),$$ then \eqref{eq:alg_Ricci_sol} holds.
    \end{proof}
    
    \begin{remark}
        The result of Theorem \ref{thm_heis3} is not new since the existence of a nilsoliton metric on $\heis_3$ is well-known in the literature (see e.g.\ \cite[Corollary 4.6]{Mil76}). In fact, it follows from this result of Milnor that \emph{any} left-invariant metric on $\heis_{3}$ is a Ricci soliton.
    \end{remark}
    
    \begin{theorem}\label{thm:c=0_soliton}
        Let $n>1$. Then the pair $(\frl,g_\rho^c)$ is a solvsoliton for $c=0$ and $\rho>0$.
    \end{theorem}

    \begin{proof}
        In this case the metric $g_\rho=g_\rho^0$ is diagonal with respect to the basis $\calB_n$ and its Ricci endomorphism is computed in Proposition \ref{Prop:gcrho_frl}: $$\ric_\rho=\diag\big(r_1\mathbbm{1}_{2n-2},r_3\mathbbm{1}_2,r_4\mathbbm{1}_{2n-2},r_2\big)=\diag\big(-2(n+2)\mathbbm{1}_{2n-2},-2\mathbbm{1}_2,-2\mathbbm{1}_{2n-2},2n\big).$$
        
        If we consider the derivation $\delta$ from Lemma \ref{Lem:DerL},  choose $\lambda=-2(n+2)$ and put 
        $$D=(2n+2)\delta = \mathrm{diag}\big(\mathbb{O}_{2n-2},(2n+2)\mathbbm{1}_2,(2n+2)\mathbbm{1}_{2n-2},2(2n+2)\big),$$ 
        then \eqref{eq:alg_Ricci_sol} holds. 
    \end{proof}

    \begin{remark}\label{rmk:general_sym_solv}
		Let $S$ be a codimension one connected Lie subgroup of the solvable Iwasawa group $AN$ of an irreducible symmetric space of non-compact type. It was shown in \cite[Theorem A]{DST21} that if $S$ contains the nilpotent part $N$, then $S$ is a Ricci soliton with respect to the metric induced by the left-invariant Einstein metric on $AN$. For the case $c=0$ we have $\SU(n,2)/\mathrm{S}(\mathrm{U}(n)\times\mathrm{U}(2))$ equipped with its symmetric metric. Thus Theorem \ref{thm:c=0_soliton} shows that our explicit computations agree with this general result.
    \end{remark}

    

    \begin{theorem}\label{thm:c>0_not_soliton}
        Let $n>1$. Then the pair $(\frl,g_\rho^c)$ is not a solvsoliton for $c>0$ and $\rho>0$.
    \end{theorem}
    
    \begin{proof}
        We have the orthogonal decomposition $\frl=\fra\oplus[\frl,\frl]$, where $[\frl,\frl]$ coincides with the nilradical of $\frl$. In this situation, Theorem \ref{thm:nilsoliton_solvsoliton} provides a characterization of the existence of a Ricci soliton on $(\frl,g_\rho^c)$. In particular, $\ad(A)$ must be a normal operator for all $A\in\fra$, that is $[\ad(A),\ad(A)^*]=0$, where $\ad(A)^*=(g_\rho^c)^{-1}\ad(A)^\top g_\rho^c$ denotes the adjoint operator of $\ad(A)$ with respect to the metric $g_\rho^c$. In Lemma \ref{Lem:adB1Rnotnormal} we will show that $$[\ad(B_1^R),\ad(B_1^R)^*]\neq0,$$ therefore we do not have a Ricci soliton.
    \end{proof}
    
    \begin{lemma}\label{Lem:adB1Rnotnormal}
        $[\ad(B_1^R),\ad(B_1^R)^*]\neq0$.
    \end{lemma}
    
    \begin{proof}
        We work in the basis $\calB_n$ given by \eqref{Eq:DefcalB}. We have computed $\ad(B_1^R)$ and $\ad(B_1^R)^*$ in the proofs of Proposition \ref{prop:non-unimodular} and Lemma \ref{lemma:curvature_term_R}, respectively. Then we compute that \begin{align*}
            [\ad(B_1^R),\ad(B_1^R)^*]&=\mathrm{diag}\big(0,0,\mathbb{O}_{2n-4},\tfrac{4c(\rho+c)}{\rho(\rho+2c)}\mathbbm{1}_2,-\tfrac{4c(\rho+c)}{\rho(\rho+2c)}\mathbbm{1}_2,\mathbb{O}_{2n-4},0\big)\\
            &\quad-\tfrac{2c}{\rho(\rho+2c)}E_{2,4n-1}-\tfrac{8c(\rho+c)^2}{\rho(\rho+2c)}E_{4n-1,2},
        \end{align*} which is zero if and only if $c=0$.
    \end{proof}
    
    \begin{remark}\label{rmk:alternative_proof}
        It was shown in \cite[Theorem 5.1]{Lau11} that if two solvsolitons are isomorphic as Lie groups, then they are isometric up to scaling. This result, together with Corollary \ref{cor:non-homothetic}, provides an alternative proof to Theorem \ref{thm:c>0_not_soliton}.
    \end{remark}
    
    Recently, it was proved by Thompson \cite[Theorem A]{Tho24} that the results of \cite{LL14} can be extended to the inhomogeneous setting. More precisely, given a unimodular solvsoliton $(S_0,g_0)$, he shows that there exists a one-parameter family of complete Ricci soliton metrics on $M=\R\times S_0$, with $\lambda<0$, that are of cohomogeneity one and exactly one of the metrics on the family is Einstein. Our results differ from the ones of Thompson. Theorem \ref{thm:c=0_soliton} gives a non-unimodular solvsoliton $(\frl,g_\rho)$ which admits a rank-one extension to a quaternionic Kähler homogeneous metric (in fact symmetric, see Remark \ref{rmk:general_sym_solv}). Whereas for $c>0$ we have a metric $g_\rho^c$ on $\frl$, which is \emph{not} a Ricci soliton by Theorem \ref{thm:c>0_not_soliton}, but such that $(0,\infty)\times L$ admits a quaternionic Kähler metric, this time of cohomogeneity one.
    
    \begin{remark}
        In \cite{QK_from_QC_14} the concept of \emph{$\Sp(n)\Sp(1)$-hypo structure} is introduced and it is shown that any oriented hypersurface of a quaternionic Kähler manifold is naturally endowed with such structure. In the same paper the authors also study how to construct quaternionic Kähler metrics on $M\times\R$ for some classes of $4n+3$-dimensional manifolds $M$. Similar constructions can be found in \cite{Fow23}, where in this case $M$ is a torus bundle over a hyperkähler manifold of dimension $4n$. However, it is not clear at all how the curvature properties of the metric $g_\rho^c$ with $c>0$ would fit in this picture. Therefore, determine which is the precise geometry of $(\bar{N}_\rho,g_\rho^c)$ when $c>0$ will be the object of a future study.
    \end{remark}
    

\bibliographystyle{amsplain}
\bibliography{biblio}

\end{document}